\documentclass[11pt]{article}
\usepackage{amsfonts}
\usepackage{graphics}
\usepackage{indentfirst}
\usepackage{color}
\usepackage{cite}
\usepackage{float}
\usepackage{caption}
\usepackage{latexsym}
\usepackage[paper=a4paper, left=1.6cm, right=1.6cm, top=1.8cm, bottom=1.6cm, headheight=5.5pt, footskip=0.8cm, footnotesep=0.8cm, centering, includefoot]{geometry}
\usepackage{amsmath}
\allowdisplaybreaks
\usepackage{amssymb}
\usepackage[colorlinks, linkcolor=red]{hyperref}
\hypersetup{urlcolor=red, citecolor=red}
\usepackage[dvips]{epsfig}
\usepackage{amscd}

\usepackage{amsthm}
\usepackage{mathrsfs}
\usepackage{verbatim}
\newtheorem{theorem}{Theorem}[section]
\newtheorem{remark}{Remark}[section]
\newtheorem{definition}{Definition}[section]
\newtheorem{lemma}{Lemma}[section]
\newtheorem{corollary}{Corollary}[section]
\newtheorem{proposition}{Proposition}[section]

\allowdisplaybreaks

\makeatletter
\@addtoreset{equation}{section}
\makeatother
\makeatletter
\@addtoreset{equation}{section}
\makeatother

\title{Global well-posedness of isentropic compressible Navier--Stokes equations with smallness on scaling-invariant quantity in a half-space
\thanks{This research was partially supported by National Natural Science Foundation of China (No. 12371227) and Fundamental Research Funds for the Central Universities (No. SWU--KU24001).}
}

\author{Lin Xu,\ Xin Zhong{\thanks{Corresponding author. E-mail addresses: mathxu@email.swu.edu.cn (L. Xu), xzhong1014@amss.ac.cn (X. Zhong).}}
	\date{}\\
	\footnotesize School of Mathematics and Statistics, Southwest University, Chongqing 400715, P. R. China}
\begin{document}
\maketitle

\begin{abstract}
	We investigate the initial-boundary value problem for the three-dimensional isentropic compressible
	Navier--Stokes equations in the upper half-space with the slip boundary
	conditions. We prove the global existence and uniqueness of strong
	solutions in the presence of vacuum and large  oscillations.
	Although scaling frameworks for compressible flows in domains with
	boundaries have been developed in several settings, the global well-posedness result  in the half-space remains far from complete.  The system with  far-field vacuum  admits a natural scaling
	structure that preserves both the half-space geometry and the slip
	boundary conditions. Motivated by this scaling, we identify the following
	\textit{scaling-invariant initial quantity}:
	$$
	\left[
	\|\rho_{0}\|_{L^{\infty}}^3
	\left(
	\frac12\|\sqrt{\rho_0} u_0\|_{L^{2}}^{2}
	+\frac{1}{\gamma-1}\|P(\rho_{0})\|_{L^{1}}
	\right)
	+\|\rho_{0}\|_{L^{\infty}}^{\frac{3-\gamma}{2}}
	\right]
	\left(
	\|\nabla u_0\|_{L^2}^2
	+\|P(\rho_{0})\|_{L^2}^2
	\right).
	$$
	Under the assumption that this quantity is sufficiently small, we establish the global well-posedness of strong solutions. This result provides a half-space counterpart of the scaling-invariant global theory for the Cauchy problem established by Wen (Adv. Math. 482 (2025), Paper No. 110628) and shows that the slip boundary condition is compatible with the system.
\end{abstract}



\setcounter{tocdepth}{2}
\tableofcontents

\section{Introduction}
\subsection{Background and motivation}

Half-space problems arise naturally in the study of fluid motions near flat
boundaries and play an important role in both mathematics and physics.
The half-space provides a basic model for wall-bounded flows,
including flows near solid surfaces, boundary layers, and fluid motions
affected by planar interfaces (see, e.g., \cite{SchlichtingGersten2017,LaugaStone2003}). It is an intermediate setting between the whole space and bounded domains: it retains the non-compactness of the whole space while incorporating boundary effects.

The present paper considers the motion of a compressible viscous fluid in
the upper half-space. More precisely, we study three-dimensional
isentropic compressible Navier--Stokes equations in the half-space:
\begin{align}\label{1}
\begin{cases}
\rho_{t}+\operatorname{div}(\rho u)=0, \\
(\rho u)_{t}+\operatorname{div}(\rho u\otimes u)
-\mu\Delta u-(\mu+\lambda)\nabla\operatorname{div}u+\nabla P(\rho)=0.
\end{cases}
\end{align}
Here
\begin{align*}
(t,x)\in\mathbb R_{+}\times\mathbb R^{3}_{+},\quad
\mathbb R^{3}_{+}\triangleq
\left\{x=(x_{1},x_{2},x_{3})\in\mathbb R^{3}:x_{3}>0\right\}.
\end{align*}
The unknowns $\rho=\rho(x,t)\geq0$ and
$u=u(x,t)=(u^{1},u^{2},u^{3})(x,t)$ denote the density and velocity field,
respectively. The pressure is given by the polytropic law
\begin{align*}
P(\rho)=A\rho^{\gamma},\quad A>0,\quad \gamma>1.
\end{align*}
Without loss of generality, we take $A=1$. The viscosity coefficients
$\mu$ and $\lambda$ satisfy the physical restrictions
\begin{align*}
\mu>0,\quad 2\mu+3\lambda\geq0.
\end{align*}

The system \eqref{1} is supplemented with the initial data
\begin{align}\label{2}
	(\rho,\rho u)(x,0)=(\rho_{0},\rho_{0}u_{0})(x),
\end{align}
and the far-field behavior
\begin{align}\label{3}
	(\rho,u)(x,t)\rightarrow(0,0)
	\quad \text{as } |x|\rightarrow\infty.
\end{align}
On the boundary
$\partial\mathbb R^{3}_{+}$, we impose the slip boundary condition
\begin{align}\label{4}
	u\cdot\nu=0,\quad
	\operatorname{curl}u\times\nu=0
	\quad \text{on } \partial\mathbb R^{3}_{+},
\end{align}
where $\nu=(0,0,-1)$ is the unit outward normal to
$\partial\mathbb R^{3}_{+}$. This condition is also referred to as the
Hodge boundary condition or the perfect-wall boundary condition in the
literature; see \cite[Chapter 4]{GN2018}.

The mathematical theory of the compressible Navier--Stokes equations has
attracted extensive attention over the past several decades. For the Cauchy
problem in the whole space, when the density is a small perturbation of a
strictly positive constant state, Matsumura--Nishida
\cite{MatsumuraNishida1980,MatsumuraNishida1983} established the global
well-posedness of smooth solutions for non-vacuum problem. Later, Hoff \cite{Hoff1995,Hoff1995a} developed a global theory
of weak solutions with discontinuous initial data. In this framework, the density is allowed to
have jump discontinuities, while it is still treated as a perturbation problem and remains require a positive lower bound for the density. Another influential approach is the critical-space theory. One may
formally introduce the parabolic scaling
\begin{align}\label{com}
	\rho^{\kappa}(t,x)\triangleq \rho(\kappa^{2}t,\kappa x),
	\qquad
	u^{\kappa}(t,x)\triangleq \kappa u(\kappa^{2}t,\kappa x),
	\qquad
	P^{\kappa}(t,x)\triangleq \kappa^{2}P(\kappa^{2}t,\kappa x).
\end{align}
Under this transformation, the compressible system keeps the same formal
structure if the pressure is regarded as an independent unknown. In the
framework of critical spaces, Danchin \cite{D2000} made a breakthrough by
establishing the global well-posedness of strong solutions for system \eqref{1} in $L^{2}$-based critical Besov spaces, for
initial density close to a constant equilibrium state. This result showed that
the critical functional framework is well adapted to the mixed
hyperbolic--parabolic structure of the compressible system. The theory was
further developed in
\cite{CharveDanchin2010,CMZ2010,H2011a,H2011b,Danchin2014}.
These works provide a fundamental global-in-time framework for compressible
viscous flows. However, they are essentially perturbative theories around a strictly
positive equilibrium state, and therefore exclude the presence of vacuum.

The presence of vacuum causes substantial additional difficulties. When
the density may vanish, the momentum equation becomes degenerate, and the
absence of a positive lower bound for the density prevents a direct use of
standard parabolic theory.  In this direction, a major breakthrough was achieved  by Lions \cite{P1998}, who
established the global existence of finite-energy weak solutions allowing
vacuum under suitable restrictions on the adiabatic exponent $\gamma$. This approach was later extended and refined in \cite{FNP01,F2004}.  Another important line of research concerns global smooth
solutions with vacuum. Huang--Li--Xin \cite{HLX12} proved the global well-posedness of classical solutions for the Cauchy problem with small initial energy, allowing vacuum and large oscillations of the initial density. Hong--Hou--Peng--Zhu \cite{HHPZ2024} further developed this direction by establishing well-posedness results allowing large initial energy in the nearly isothermal case. Subsequent works further
developed the small-energy method and clarified the roles of the \textit{ effective
viscous flux} and the uniform upper bound of the density; see \cite{LiXin2019,HuangLi2012} and the references therein.

We emphasize that the parabolic scaling in \eqref{com} is not
compatible with the  pressure law $P(\rho)=\rho^{\gamma}$. The  pressure term fixes a specific critical scaling. Lei--Xin \cite{LX19} observed the following scaling-invariant property:
\begin{align}\label{scaling}
	\rho^{\ell}(x,t)
	 \triangleq
	\ell^{\frac{2}{\gamma+1}}
	\rho\left(\ell x,\ell^{\frac{2\gamma}{\gamma+1}}t\right),\quad
	u^{\ell}(x,t)
	 \triangleq
	\ell^{\frac{\gamma-1}{\gamma+1}}
	u\left(\ell x,\ell^{\frac{2\gamma}{\gamma+1}}t\right),
\end{align}
where $\ell>0$ is the scaling parameter. The dependence of the scaling
exponents on $\gamma$ reflects the nonlinear structure of the pressure;
for more details, see \cite[Appendix]{XZ20262}.
Based on this scaling, Wen \cite{W25} first constructed the following
scaling-invariant initial quantity $S_{0}(\rho_{0},u_{0})$:
\begin{align}\label{scaling1}
	\|\rho_{0}\|_{L^{\infty}}^{3}
	\left(
	\|\sqrt{\rho_{0}}u_{0}\|_{L^{2}}^{2}
	+\|\rho_{0}^{\gamma}\|_{L^{1}}
	\right)
	\left(
	\|\nabla u_{0}\|_{L^{2}}^{2}
	+\|\rho_{0}^{\gamma}\|_{L^{2}}^{2}
	\right)
	\left[
	1+
	\|\rho_{0}\|_{L^{\infty}}^{3+\gamma}
	\left(
	\|\sqrt{\rho_{0}}u_{0}\|_{L^{2}}^{2}
	+\|\rho_{0}^{\gamma}\|_{L^{1}}
	\right)^{2}
	\right].
\end{align}
Here the scaling-invariant means that
\begin{align*}
S_{0}(\rho_{0},u_{0})=S_{0}(\rho^{\ell}_{0},u^{\ell}_{0})
\end{align*}
for the rescaled initial data. Provided that this scaling-invariant
initial quantity is sufficiently small, Wen \cite{W25} proved the global
existence of strong solutions for  far-field vacuum
Cauchy problem of system \eqref{1}. Recently, Xu--Zhong \cite{XZ20261} introduced another
scaling-invariant initial quantity of the form
\begin{align}\label{scaling2}
	C_{\mu,\lambda,\gamma}
	\left(1+\|\rho_{0}\|_{L^{\infty}}\right)^{3}
	\left(
	\|\sqrt{\rho_{0}}u_{0}\|_{L^{2}}^{2}
	+\|\rho_{0}^{\gamma}\|_{L^{1}}
	\right)
	\left[
	\mu\|\nabla u_{0}\|_{L^{2}}^{2}
	+(\mu+\lambda)\|\operatorname{div}u_{0}\|_{L^{2}}^{2}
	+\frac{\|\rho_{0}^{\gamma}\|_{L^{2}}^{2}}{2\mu+\lambda}
	\right],
\end{align}
where $C_{\mu,\lambda,\gamma}$ is a parameter depending on
$ \mu$, $\lambda$, and $\gamma$. In particular, the smallness is independent of the initial data and the physical parameters of the system. The above two results exhibit two different scaling-invariant initial
quantities for the same far-field vacuum problem. This indicates that the
choice of a scaling-invariant smallness quantity is not unique. This
naturally leads to the following questions:

\textbf{Q1: Can we construct new scaling-invariant quantities for the far-field vacuum problem?}

We  turn to initial-boundary value problems. Such problems are indispensable for describing viscous compressible
fluids in the presence of solid walls, obstacles, or confined containers.
The boundary creates additional
difficulties in the analysis and estimates, and the choice of boundary condition plays an essential role. The present paper focuses on half-space problems. For the no-slip or
Dirichlet boundary condition, Perepelitsa \cite{P2014} obtained
global weak solutions in a half-space near a equilibrium, where the
density stays strictly positive. In the  critical-space
framework, Danchin--Mucha \cite{DM2009} developed a critical
functional setting for the inhomogeneous incompressible Navier--Stokes
equations in the half-space, and Danchin--Zhang
\cite{DZ2014} further studied the corresponding problem with
only bounded density. Recently, Chikami--Ogawa--Shimizu
\cite{ChikamiOgawaShimizu2025} established local well-posedness in
critical spaces for system \eqref{1} in a nearly
half-space, using the
Lagrangian transform, boundary flattening and endpoint
$L^{1}$-maximal regularity. Ogawa--Shimizu \cite{OgawaShimizu2026} further considered the free boundary problem  in critical spaces. More recently, Wang--Zhong
\cite{WZ2026} studied the Dirichlet problem for \eqref{1} with vacuum and
unbounded density under a suitably small initial total energy.

We next recall the results under Navier or slip boundary conditions, which
are more closely related to our setting. Hoff \cite{Hoff2005} proved
the global existence of weak solutions in a half-space under the assumption that the initial energy is suitably
small. Later, Duan \cite{Duan2012} developed the global
theory of classical solutions for \eqref{1} in a half-space. In that work, the initial density is allowed to contain vacuum
and have large oscillations, provided that the initial energy is suitably
small. Subsequently, Song--Zhang \cite{SongZhang2021} considered the case of large external forces.
Duan--Huang \cite{DuanHuang2022} established the local existence of weak
solutions with higher regularity for \eqref{1}
in the presence of vacuum. Let us also mention some results on large-time behavior and singular
limits in half-spaces.
  The decay and
asymptotic behavior of solutions for \eqref{1} near a positive constant equilibrium has been studied in \cite{KageiKobayashi2002,KageiKobayashi2005}.   Wang--Wang \cite{WangWang2021} investigated the large-time behavior toward planar rarefaction waves in a three-dimensional half-space, while
Chang--Liu--Xu \cite{ChangLiuXu2025} established the nonlinear stability
of planar shock waves. Very recently,
Wang--Wu--Zhong \cite{WWZ2026} analyzed the incompressible
limit for the half-plane flows with large initial data and
vacuum. Gao--Peng--Wu--Yao \cite{GPWY2026} considered the anisotropic viscosity and established
global-in-time uniform regularity together with the global vanishing
vertical viscosity limit for small perturbations of a constant equilibrium
state. These works are closely related to decay and limiting mechanisms in
boundary problems.

Despite these developments, the global strong solution theory in the
half-space with vacuum is still far from complete. \textit{In particular, the
smallness assumptions in the existing global results are usually formulated
in terms of the initial energy (see \cite{Hoff2005,Duan2012}) and are not designed to reflect the
intrinsic scaling structure of the equations.}
  We also note that the critical regularity theory for system \eqref{1} in domains with boundary is much less developed
  than that for the Cauchy problem. For problems with boundary, important progress has been made
in bounded domains and nearly half-spaces; see, for instance,
\cite{DanchinTolksdorf2023,ChikamiOgawaShimizu2025}. However, a global
scaling-critical well-posedness theory for \eqref{1} in the half-space still seems to be
unavailable. This observation naturally leads to the second question
addressed in this paper:

\textbf{Q2: Can we establish a global well-posedness theory for the scaling  half-space problem?}

Motivated by the above two questions, the present paper gives an affirmative answer in the half-space.
We first observe that the scaling of \eqref{1} is
compatible with both the flat half-space geometry and the slip boundary
condition. More precisely, the scaling preserves the domain
$\mathbb R^{3}_{+}$ and keeps the boundary condition \eqref{4} unchanged. This observation motivates us to formulate a global theory under a scaling-invariant smallness condition. Then we construct \textit{a new
scaling-invariant initial quantity}, different from \eqref{scaling1} and \eqref{scaling2} in the previous Cauchy problem results, and prove the global existence and uniqueness of strong solutions provided that this quantity is sufficiently small.

\subsection{Main result}
Before stating our results, we introduce some notation that will be used
frequently throughout the paper. For simplicity, we use the shorthand
notation
\begin{align*}
\bar{\rho}\triangleq \|\rho_{0}\|_{L^{\infty}},\quad
\Omega\triangleq \mathbb R^{3}_{+},\quad
\dot{u}\triangleq u_{t}+u\cdot\nabla u,\quad
f_{i}\triangleq \partial_{i}f=\frac{\partial f}{\partial x_{i}}.
\end{align*}
The energy functional is defined by
\begin{align*}
E(t)\triangleq
\int_{\Omega}
\bigg(
\frac{1}{2}\rho |u|^{2}+\frac{1}{\gamma-1}P(\rho)
\bigg)(x,t)\,dx.
\end{align*}
We also set
\begin{align*}
D(t)\triangleq
\int_{\Omega}\big(|\nabla u|^{2}
+P^{2}\big)(x,t)\,dx.
\end{align*}
Moreover, the \textit{effective viscous flux} is defined by
\begin{align*}
F\triangleq
(2\mu+\lambda)\operatorname{div}u-P.
\end{align*}

For an integer $k\geq0$ and $1\leq q\leq\infty$, we define
\begin{align*}
D^{k,q}(\Omega)
=
\big\{
u\in L^{1}_{\mathrm{loc}}(\Omega):
\|\nabla^{k}u\|_{L^{q}(\Omega)}<\infty
\big\},
\quad
\|u\|_{D^{k,q}}=\|\nabla^{k}u\|_{L^{q}(\Omega)}.
\end{align*}
We also use the notation
\begin{align*}
\|u\|_{L^{q}}=\|u\|_{L^{q}(\Omega)},\quad
W^{k,q}(\Omega)=L^{q}(\Omega)\cap D^{k,q}(\Omega),\quad
D^{k}(\Omega)=D^{k,2}(\Omega),\quad
H^{k}(\Omega)=W^{k,2}(\Omega).
\end{align*}

The strong solutions considered in this paper are defined as follows.

\begin{definition}\label{def}
	Let $T>0$ and $q\in(3,6)$. Assume that the initial data
	$(\rho_{0}\geq 0,u_{0})$  satisfies
	\begin{align}\label{sou}
		\rho_{0}\in H^{1}\cap W^{1,q},
		\qquad
		u_{0}\in D^{1}.
	\end{align}
	A pair $(\rho,u)$ is called a strong solution to
	\eqref{1}--\eqref{4} in $\Omega\times[0,T]$ if it satisfies the
	following regularity properties:
	\begin{align}\label{defs}
		\begin{cases}
			\rho\in C\big([0,T];L^{2}\big)
			\cap L^{\infty}\big(0,T;H^{1}\cap W^{1,q}\big),
			\quad
			\rho_{t}\in L^{\infty}\big(0,T;L^{2}\big), \\[1mm]
			\rho u\in C\big([0,T];L^{2}\big),
			\quad
			u\in L^{\infty}\big(0,T;D^{1}\big)
			\cap L^{2}\big(0,T;D^{2}\big),
			\quad
			\sqrt{\rho}u_{t}\in L^{2}\big(0,T;L^{2}\big), \\[1mm]
			\sqrt{t}u\in L^{\infty}\big(0,T;D^{2}\big)
			\cap L^{2}\big(0,T;D^{2,q}\big),
			\quad
			\sqrt{t}u_{t}\in L^{2}\big(0,T;D^{1}\big).
		\end{cases}
	\end{align}
	Moreover, $(\rho,u)$ satisfies \eqref{1}--\eqref{4} almost everywhere in
	$\Omega\times(0,T)$. If the
	above properties hold for any $T>0$, then $(\rho,u)$ is called a
	global strong solution.
\end{definition}

System \eqref{1} is invariant under the scaling
\eqref{scaling} for any $\ell>0$. Moreover, this transformation preserves
both the upper half-space $\mathbb R^{3}_{+}$, the slip boundary condition \eqref{4}, and the far-field vacuum
condition \eqref{3}. We next verify that the slip boundary condition \eqref{4} is also invariant
under \eqref{scaling}. If $x\in\partial\mathbb R^3_+$, then $x_3=0$, and
hence $(\ell x)_3=\ell x_3=0$. Thus the boundary
$\partial\mathbb R^3_+$ is mapped onto itself by the spatial dilation
$x\mapsto \ell x$, and the outward unit normal remains unchanged. On the
boundary, we have
\begin{align*}
u^\ell(x,t)\cdot\nu
=
\ell^{\frac{\gamma-1}{\gamma+1}}
u\left(\ell x,\ell^{\frac{2\gamma}{\gamma+1}}t\right)\cdot\nu.
\end{align*}
Since $\ell x\in\partial\mathbb R^3_+$, the condition
$u\cdot\nu=0$ on $\partial\mathbb R^3_+$ implies
$u^\ell\cdot\nu=0$ on $\partial\mathbb R^3_+$. By the chain rule,
\begin{align*}
\operatorname{curl}u^\ell(x,t)\times\nu
=
\ell^{\frac{2\gamma}{\gamma+1}}
\left[
\operatorname{curl}u
\left(\ell x,\ell^{\frac{2\gamma}{\gamma+1}}t\right)
\times\nu
\right].
\end{align*}
Again, since $\ell x\in\partial\mathbb R^3_+$, the boundary condition
$\operatorname{curl}u\times\nu=0$ on $\partial\mathbb R^3_+$ implies
\begin{align*}
\operatorname{curl}u^\ell\times\nu=0
\qquad \text{on } \partial\mathbb R^3_+.
\end{align*}
Hence, the homogeneous slip boundary condition \eqref{4} is preserved under
the scaling \eqref{scaling}.

We are now in a position to state our   main result.

\begin{theorem}\label{th1}
	Let $q\in(3,6)$  and assume that the initial data
	$(\rho_{0}\geq0,u_{0})$  satisfies
	\begin{align}\label{xz}
		\rho_{0}\in L^{\gamma}\cap H^{1}\cap W^{1,q},
		\quad
		u_{0}\in D^{1}, \quad \sqrt{\rho_{0}}u_{0}\in L^{2}.
	\end{align}
	There exists a positive constant $\varepsilon_{0}$ depending only
	on the parameters $\mu$ and $\lambda$ such that if
	\begin{align}\label{sca}
		\left[
		\bar\rho^3
		\left(
		\frac12\|\sqrt{\rho_{0}}u_{0}\|_{L^{2}}^{2}
		+\frac{1}{\gamma-1}\|P(\rho_{0})\|_{L^{1}}
		\right)
		+\bar\rho^{\frac{3-\gamma}{2}}
		\right]
		\left(
		\|\nabla u_0\|_{L^2}^2
		+\|P(\rho_{0})\|_{L^2}^2
		\right)
		\leq \varepsilon_{0},
	\end{align}
	then the initial-boundary value problem \eqref{1}--\eqref{4} admits a
	unique global strong solution $(\rho,u)$.
\end{theorem}

\begin{remark}
It should be noted that the smallness condition \eqref{sca} is independent of the initial data. Moreover, it is worth mentioning that the quantity appeared in the left-hand side of \eqref{sca} is invariant under the scaling \eqref{scaling}.
\end{remark}

\begin{remark}
Compared with the quantities \eqref{scaling1} and \eqref{scaling2} used in \cite{W25} and \cite{XZ20261}, respectively, the scaling-invariant quantity \eqref{sca} has a simpler structure since it does not contain the additional higher-power terms
\begin{align*}
1+\|\rho_{0}\|_{L^{\infty}}^{3+\gamma}
\left(\|\sqrt{\rho_{0}}u_{0}\|_{L^{2}}^{2}
+\|\rho_{0}^{\gamma}\|_{L^{1}}
\right)^{2},\quad \left(1+\|\rho_{0}\|_{L^{\infty}}\right)^{3}.
\end{align*}
When $1<\gamma<\frac{3}{2}$,  \eqref{sca} shows that the initial
energy is not required to be small; instead, the allowable size of the initial energy depends on the upper bound of the initial density. In particular, a
large initial energy may be allowed when the density upper bound is
sufficiently small and the remaining factors in \eqref{sca} are properly
controlled. This feature is consistent with the recent result of
\cite{HHPZ2024}, where the authors proved the global existence and
uniqueness of classical solutions with large initial energy in the nearly
isothermal regime, namely for $\gamma$ sufficiently close to $1$.
\end{remark}

\begin{remark}
It should be pointed out that the no-slip boundary condition
\begin{align*}
u\big|_{\partial\mathbb R^{3}_{+}}=0
\end{align*}
is compatible with the scaling of the far-field vacuum problem.
Critical functional frameworks in the half-space with no-slip type
boundary conditions have been developed, for instance, by Danchin--Mucha
\cite{DM2009} and Danchin--Zhang \cite{DZ2014} for the inhomogeneous
Navier--Stokes equations. Nevertheless, for the compressible
Navier--Stokes equations with vacuum, the no-slip boundary condition
creates additional difficulties in the higher-order estimates, since it
does not yield simple homogeneous boundary conditions for the vorticity or the	 \textit{effective viscous flux}.
\end{remark}

\begin{remark}
The scaling compatibility of the slip boundary condition \eqref{4} is essential in the case of far-field vacuum. This compatibility is no longer valid for the Robin-type slip boundary condition with a fixed friction coefficient
\begin{align}\label{z1}
u\cdot\nu=0,\qquad
\operatorname{curl}u\times\nu=-k u
\quad \text{on } \partial\mathbb R^{3}_{+},
\end{align}
where $k>0$ is a constant. Moreover, this boundary condition
produces additional boundary terms in the higher-order estimates.  Duan \cite{Duan2012} treated a related Robin-type slip boundary condition with far-field non-vacuum, while whether one can establish the global strong-solution theory with far-field vacuum under the Robin-type slip boundary condition \eqref{z1} remains unclear.
\end{remark}

\begin{remark}
	Let us point out a useful geometric feature of the half-space. For general
	bounded domains with slip boundary conditions, higher-order estimates may
	produce nontrivial boundary terms. This is mainly because the outward normal
	vector $n=n(x)$ depends on the boundary point. For instance,
	differentiating the boundary condition $u\cdot n=0$ along the flow
	usually gives
\begin{align*}
	\dot u\cdot n
	=
	u_t\cdot n+u\cdot\nabla u\cdot n
	=
	-u\cdot\nabla n\cdot u
	\qquad \text{on } \partial\Omega.
\end{align*}
	Thus, in a general bounded domain, boundary integrals involving
	$\dot u\cdot n$ do not vanish and have to be estimated through the
	curvature of the boundary. In the present half-space case, the
	boundary is flat and
	\begin{align*}
	\nu=(0,0,-1),\quad \nabla \nu=0
	\quad \text{on } \partial\mathbb R^{3}_{+}.
\end{align*}
	Consequently, the curvature contribution disappears and one obtains
	\begin{align*}
	\dot u\cdot \nu=0
	\qquad \text{on } \partial\mathbb R^{3}_{+}.
\end{align*}
	Moreover, since
	$\partial\mathbb R^{3}_{+}=\{x\in\mathbb R^{3}:x_{3}=0\}$, the boundary
	condition \eqref{4} is equivalent to
	\begin{align*}
		u^{3}=0,\qquad
		\partial_{x_{3}}u^{1}=\partial_{x_{3}}u^{2}=0
		\qquad \text{on } \partial\mathbb R^{3}_{+}.
	\end{align*}
	Therefore, the boundary terms which usually arise in general bounded
	domains vanish in the half-space.
\end{remark}

\begin{remark}
It should be emphasized that the smallness condition in Theorems \ref{th1} depends only on the viscosity coefficients $\mu$ and $\lambda$. A natural further question is whether one can derive estimates that are uniform with respect to the viscosity parameters and then study singular limits involving these coefficients. We leave this interesting problem for future investigation.
\end{remark}

\subsection{Strategy of the proof}
We briefly explain the main ideas of the proofs. For the initial-boundary
value problem, the presence of the boundary brings additional analytical
difficulties. The boundary condition \eqref{4}   plays an essential role in the elliptic estimates for
the \textit{effective viscous flux} and the vorticity. The proofs are based
on a combination of scaling-invariant lower-order estimates, time-weighted
higher-order estimates, and a standard continuation argument. The main
difficulty is to keep the smallness assumption consistent with the
intrinsic scaling structures of the corresponding system, allowing
vacuum and large oscillations.

One of the main technical points is to establish elliptic estimates for the
\textit{effective viscous flux} and the vorticity in the half-space; see
Lemma \ref{FF}. Compared with the Cauchy problem
\cite{W25,XZ20261,XZ20262}, the presence of the boundary makes the
vorticity estimates more delicate. The slip boundary condition \eqref{4}
implies that two vorticity components satisfy homogeneous boundary
conditions, and hence they can be estimated by the standard $L^{p}$-theory
in the half-space (see \eqref{j1} and \eqref{j2}). The remaining vorticity component does not satisfy such
a boundary condition. To overcome this difficulty, we exploit the flat
geometry of the half-space and derive an elliptic equation in the horizontal
variables. Together with the two-dimensional Calder{\'o}n--Zygmund estimate,
this gives the desired $L^{p}$-control of the full vorticity (see \eqref{curl}). Once the vorticity is controlled, the estimate for the \textit{effective viscous flux} follows from the decomposition of the momentum equation \eqref{demo}. These boundary-compatible elliptic estimates are crucial for controlling the material derivative and closing the global a priori estimates.

For scaling-invariant lower-order estimates, the argument is based on a bootstrap procedure; see
Proposition \ref{0pro}. The first key step is to obtain a uniform upper
bound for the density. This is achieved by applying Zlotnik's inequality (see Lemma \ref{zo})
along particle trajectories. In this step, the pressure produces an
additional difficult term in \eqref{kun}, namely
\begin{align*}
\int_{t_1}^{t_2}
\|\nabla u\|_{L^2}
\|P\|_{L^6}\,d\tau.
\end{align*}
To handle this term, we introduce the scaling-invariant control quantity
\begin{align}\label{f1}
	\bar\rho^{\frac{3-\gamma}{2}}D(t),
\end{align}
which makes the pressure contribution sufficiently small so that it can be
absorbed into the coefficient $N_1$ in Zlotnik's inequality. In order to obtain the uniform upper bound for the
density, we also need the smallness of another scaling-invariant quantity
\begin{align}\label{ff}
	\bar\rho^3E(0)D(t),
\end{align}
see \eqref{III1-new}, \eqref{III2-first-new}, and \eqref{xsca1}.

We then derive a closed estimate for $D(t)$. Taking the $L^2$-inner
product of \eqref{1}$_2$ with $u_t$, using the boundary condition
\eqref{bian2}, and exploiting the \textit{effective viscous flux}, we
obtain the estimate \eqref{2.12}. To control the term
$\|P\|_{L^{3}}^{3}$ and obtain a closed estimate for $D(t)$, we use the
pressure equation to get
\begin{align}\label{ffx2}
	\frac{d}{d t}\|P\|_{L^{p}}^{p}
	+\frac{p \gamma-1}{2(2 \mu+\lambda)}
	\|P\|_{L^{p+1}}^{p+1}
	\leq C \|F\|_{L^{p+1}}^{p+1}.
\end{align}
In particular, \eqref{ffx2} implies
\begin{align*}
\frac{d}{dt}\|P\|_{L^2}^{2}
+\frac{2\gamma-1}{2(2\mu+\lambda)}
\|P\|_{L^{3}}^{3}
\leq C\|F\|_{L^{3}}^{3},
\end{align*}
which combined with \eqref{2.12} yields the desired estimate
\eqref{2.15} provided that the scaling-invariant quantity \eqref{ff} is
sufficiently small (see \eqref{sca3}). The two scaling-invariant quantities
\eqref{f1} and \eqref{ff} are used together to obtain the uniform upper
bound for the density in the bootstrap argument, while \eqref{2.15} together with the smallness of the initial scaling-invariant quantity
\eqref{sca} closes the bootstrap argument.

We next establish the time-weighted higher-order estimates. These estimates
are used to justify the time continuity of the strong solution and to
ensure that the local solution can be extended. A crucial estimate is
\begin{align}\label{key1}
	\int_{0}^{T}\|\nabla u\|_{L^{\infty}}\,dt<\infty,
\end{align}
which is needed for the propagation of the $H^{1}\cap W^{1,q}$
regularity of the density. The logarithmic  type
inequality \eqref{g4} plays an important role in deriving
\eqref{key1}. Applying $(-\Delta)^{-1}\operatorname{div}$ to the momentum
equation gives the identity
\begin{align*}
P
=
(-\Delta)^{-1}\operatorname{div}(\rho u)_{t}
+(-\Delta)^{-1}\operatorname{div}\operatorname{div}(\rho u\otimes u)
+(2\mu+\lambda)\operatorname{div}u.
\end{align*}
This along with the \textit{a priori} estimates already established
allows us to obtain the $L_{t}^{2}L_{x}^{2}$-norm of $P$ (see Proposition \ref{wen}). The estimate \eqref{ffx2} is crucial in
controlling the higher integrability of the pressure and in closing the
time-weighted estimates involving $\sqrt{\rho}\dot u$ and
$\nabla\dot u$ (see Lemma \ref{lem52}). Here the boundary condition \eqref{4} plays a key role. It implies that
\begin{align*}
\dot u\cdot\nu=\dot u^3=0,\qquad
\partial_{x_3}\dot u^1=\partial_{x_3}\dot u^2=0
\quad \text{on } \partial\Omega.
\end{align*}
Thus, $\dot u$ satisfies the same homogeneous slip condition on the
boundary. This property allows us to handle the boundary terms appearing in
\eqref{kun1}, \eqref{J2-est}, and \eqref{kk3}, such as
\begin{align*}
\partial_3 u^k\partial_k u^j
+\partial_k u^3\partial_k u^j
-\operatorname{div}u\partial_3 u^j,\quad
-\dot u^j\partial_j u^3\operatorname{div}u,\quad
P u^k\partial_k\dot u^j\nu_j.
\end{align*}
Finally, we apply a standard continuation argument to prove that the strong
solution exists globally in time.

The remainder of this paper is organized as follows. Section \ref{sec2}
collects the elliptic estimates in the half-space and several auxiliary
inequalities used throughout the paper. Section \ref{sec3} establishes the
scaling-invariant a priori estimates and the time-weighted estimates.
Finally, Section \ref{sec4} completes the proof of Theorem \ref{th1}.

\section{Preliminaries}\label{sec2}

In this section we collect some useful results that will be frequently
used in later sections. We begin with the local well-posedness of strong
solutions (see, e.g., \cite{GLLZ20,H21}).

\begin{proposition}[Local well-posedness]\label{local}
	Assume that the initial data $(\rho_{0},u_{0})$ satisfy \eqref{sou}.
	Then there exists a time $T>0$ such that the problem
	\eqref{1}--\eqref{4} admits a unique strong solution in
	$\Omega\times(0,T]$.
\end{proposition}

We recall the classical Gagliardo--Nirenberg inequality in the half-space (see \cite[Theorem, p.125]{N1959}).

\begin{lemma}\label{og}
	Let $m$ be a positive integer, and let $p$, $q$, $r\in[1,\infty]$. If
	$u\in L^{q}\cap D^{m,p}$, then for any integer $k\in[0,m)$,
	\begin{align*}
	\|\nabla^{k}u\|_{L^{r}}
	\leq
	C\|\nabla^{m}u\|_{L^{p}}^{\alpha}
	\|u\|_{L^{q}}^{1-\alpha},
\end{align*}
	where
	\begin{align*}
	\frac{1}{r}-\frac{k}{3}
	=
	\alpha\left(\frac{1}{p}-\frac{m}{3}\right)
	+(1-\alpha)\frac{1}{q},
\end{align*}
	with
	\begin{align*}
	\left\{
	\begin{array}{ll}
		\alpha\in\left[\frac{k}{m},1\right),
		& \text{if } p\in(1,\infty)
		\text{ and } m-k-\frac{3}{p}\in\mathbb N, \\[2mm]
		\alpha\in\left[\frac{k}{m},1\right],
		& \text{otherwise.}
	\end{array}
	\right.
\end{align*}
	When $q=\infty$, $k=0$, and $mp<3$, one additionally assumes either
	$u(x)\to0$ as $|x|\to\infty$, or $u\in L^{z}$ for some finite
	$z\geq1$. The constant $C$ depends only on $m$, $k$, $p$, $q$, and $\alpha$.
	
	In particular, the following estimates will be used frequently:
	\begin{align*}
	\|u\|_{L^{\infty}}
	\leq
	C
	\|u\|_{L^{q}}^{\frac{q(p-3)}{3p+q(p-3)}}
	\|\nabla u\|_{L^{p}}^{\frac{3p}{3p+q(p-3)}},
\end{align*}
	for $u\in L^{q}\cap W^{1,p}$, $1<q<\infty$, and $3<p<\infty$, and
	\begin{align*}
	\|u\|_{L^{r}}
	\leq
	C
	\|u\|_{L^{2}}^{\frac{6-r}{2r}}
	\|\nabla u\|_{L^{2}}^{\frac{3r-6}{2r}},
	\qquad
	2\leq r\leq6.
\end{align*}
	As a special case, taking $k=0$, $m=1$, $\alpha=1$, and $p=2$, we
	obtain
	\begin{align*}
	\|u\|_{L^{6}}
	\leq
	C\|\nabla u\|_{L^{2}}.
\end{align*}
\end{lemma}

Next, we consider the Poisson equation with the homogeneous Dirichlet
boundary condition
\begin{align}\label{2.4}
\begin{cases}
		\Delta u=\partial_{x_{i}}g,
		& x\in\Omega, \\
		u=0,
		& x\in\partial\Omega.
	\end{cases}
\end{align}
The following elliptic estimates can be found in \cite[Lemma 4]{P2014}.

\begin{lemma}\label{zell}
	The following assertions hold.
	\begin{itemize}
		\item[(i)]
		If $g\in L^{2}$, then problem \eqref{2.4} admits a unique weak solution.
		
		\item[(ii)]
		If $g\in L^{2}\cap L^{p}$ for some $p\in(1,\infty)$, then there
		exists a constant $C(p)>0$ such that
		\begin{align*}
		\|\nabla u\|_{L^{p}}
		\leq
		C(p)\|g\|_{L^{p}}.
	\end{align*}
		
		\item[(iii)]
		If $g\in L^{2}$ and $\partial_{x_{i}}g\in L^{p}$ for some
		$p\in(1,\infty)$, then there exists a constant $C(p)>0$ such that
		\begin{align*}
		\|\nabla^{2}u\|_{L^{p}}
		\leq
		C(p)\|\partial_{x_{i}}g\|_{L^{p}}.
	\end{align*}
	\end{itemize}
\end{lemma}

Consider the following boundary value problem for the Lam\'e operator:
\begin{align}\label{dd}
	\begin{cases}
\mu\Delta u+(\mu+\lambda)\nabla\operatorname{div}u=K,
		& \text{in } \Omega, \\
u^{3}=\partial_{3}u^{1}=\partial_{3}u^{2}=0,
		& \text{on } \partial\Omega.
	\end{cases}
\end{align}
The corresponding Lam\'e estimates in the half-space follow from the
standard whole-space estimates by using the even--odd extension compatible
with the slip boundary condition \eqref{4}; see
\cite[Proposition 2.1]{SWZ2011}.

\begin{lemma}\label{ellp}
	Let $q\in(1,\infty)$, and let $u$ be a solution of problem \eqref{dd}.
	Then there exists a constant $C>0$, depending only on $\lambda$, $\mu$,
	and $q$, such that the following estimates hold:
	\begin{itemize}
		\item[(i)]
		If $K\in L^{q}$, then
		\begin{align*}
		\|\nabla^{2}u\|_{L^{q}}
		\leq
		C\|K\|_{L^{q}}.
	\end{align*}
		
		\item[(ii)]
		If $K\in W^{-1,q}$, namely $K=\operatorname{div}k$ with
		$k=(k_{ij})_{3\times3}$ and $k_{ij}\in L^{q}$, then
	\begin{align*}
		\|\nabla u\|_{L^{q}}
		\leq
		C\|k\|_{L^{q}}.
	\end{align*}
	\end{itemize}
\end{lemma}

The following Hodge-type decomposition is given in
\cite[Theorem 2.1]{WW92}.

\begin{lemma}\label{ho}
	Let $1<p<\infty$. For $u\in W^{1,p}$ satisfying
	$(u\cdot\nu)\big|_{\partial\Omega}=0$, there exists a positive constant
	$C$ such that
	\begin{align*}
	\|\nabla u\|_{L^{p}}
	\leq
	C\left(
	\|\operatorname{div}u\|_{L^{p}}
	+
	\|\operatorname{curl}u\|_{L^{p}}
	\right).
\end{align*}
\end{lemma}

For the \textit{effective viscous flux} $F$ defined by
\begin{align*}
F\triangleq(2\mu+\lambda)\operatorname{div}u-P,
\end{align*}
the vorticity $\operatorname{curl}u$, and $\nabla u$, we have the following estimates.
\begin{lemma}\label{FF}
	Let $(\rho,u)$ be a smooth solution to \eqref{1}--\eqref{4}. Then, for
	any $2\leq p\leq6$, there exists a positive constant $C$ depending
	only on $p$, $\lambda$, and $\mu$ such that
	\begin{gather}
		\|\nabla F\|_{L^p}
		+\|\nabla \operatorname{curl}u\|_{L^p}
		\leq
		C\|\rho\dot u\|_{L^p}, \label{41}\\
		\|\nabla u\|_{L^p}
		\leq
		C\big(
		\|F\|_{L^p}
		+\|P \|_{L^p}
		+\|\operatorname{curl}u\|_{L^p}
		\big), \label{40}\\
		\|F\|_{L^p}
		+\|\operatorname{curl}u\|_{L^p}
		\leq
		C\big(
		\|\nabla u\|_{L^{2}}
		+\|P \|_{L^{2}}
		\big)^{\frac{6-p}{2p}}
		\|\rho\dot u\|_{L^{2}}^{\frac{3p-6}{2p}},
		\label{42}\\
		\|\nabla u\|_{L^{p}}
		\leq
		C\|\rho\dot u\|_{L^{2}}^{\frac{3p-6}{2p}}
		\|\nabla u\|_{L^{2}}^{\frac{6-p}{2p}}
		+
		C\|\rho\dot u\|_{L^{2}}^{\frac{3p-6}{2p}}
		\|P \|_{L^{2}}^{\frac{6-p}{2p}}
		+
		C\|P \|_{L^{p}}.
		\label{43}
	\end{gather}
\end{lemma}

\begin{proof}
	We introduce the vorticity tensor
\begin{align*}
	\omega^{j,k}\triangleq u^j_{x_k}-u^k_{x_j},
	\qquad 1\leq j,~k\leq3.
\end{align*}
	More precisely,
\begin{align*}
	\omega =
	\begin{pmatrix}
		0
		&
		u^1_{x_2}-u^2_{x_1}
		&
		u^1_{x_3}-u^3_{x_1}
		\\[1mm]
		u^2_{x_1}-u^1_{x_2}
		&
		0
		&
		u^2_{x_3}-u^3_{x_2}
		\\[1mm]
		u^3_{x_1}-u^1_{x_3}
		&
		u^3_{x_2}-u^2_{x_3}
		&
		0
	\end{pmatrix}.
\end{align*}
This tensor contains the same information as the vorticity
\begin{align*}
	(\operatorname{curl}u)^1=u^3_{x_2}-u^2_{x_3},\qquad
	(\operatorname{curl}u)^2=u^1_{x_3}-u^3_{x_1},\qquad
	(\operatorname{curl}u)^3=u^2_{x_1}-u^1_{x_2}.
\end{align*}
Thus, $\omega$ is equivalent to $\operatorname{curl}u$.
Fixed $j$, set
\begin{align*}
\omega^{j,k}_{x_k}\triangleq\sum_{k=1}^{3}\partial_{x_k}\omega^{j,k}
=\sum_{k=1}^{3}\partial_{x_k}(u^j_{x_k}-u^k_{x_j})
=\Delta u^j-\partial_{x_j}\operatorname{div}u.
\end{align*}
Then one sees that the $j$-th component of \eqref{1}$_2$ can be rewritten as
\begin{align}\label{demo}
\rho\dot u^j=F_{x_j}+\mu\omega^{j,k}_{x_k}.
\end{align}
	
Next, we derive the $L^p$-estimate for the vorticity tensor. On the flat
boundary $\partial\Omega=\{x_3=0\}$, the slip boundary condition
\eqref{4} is equivalent to
\begin{align*}
	u^3=u^1_{x_3}=u^2_{x_3}=0
	\qquad \text{on } \partial\Omega.
\end{align*}
Since $u^3\big|_{\partial\Omega}=0$, its tangential derivatives also vanish there. Hence
\begin{align*}
	\omega^{1,3}=0,\qquad \omega^{2,3}=0
	\qquad \text{on } \partial\Omega.
\end{align*}
Let $H\triangleq\omega^{1,3}$. Taking $\partial_{x_3}$ of \eqref{demo} with
	$j=1$, taking $\partial_{x_1}$ of \eqref{demo} with $j=3$, and
	subtracting the two identities, we obtain that
\begin{align*}
	\mu\Delta H
	=
	(\rho\dot u^1)_{x_3}-(\rho\dot u^3)_{x_1}
	\qquad \text{in } \Omega.
\end{align*}
	Since $H=0$ on $\partial\Omega$, the $L^p$-estimate for the
	Dirichlet problem in the half-space (Lemma \ref{zell}) gives
\begin{align}\label{j1}
	\|\nabla \omega^{1,3}\|_{L^p}
	\leq
	C\|\rho\dot u\|_{L^p}.
\end{align}
	The same argument applied to $\omega^{2,3}$ yields that
	\begin{align}\label{j2}
	\|\nabla \omega^{2,3}\|_{L^p}
	\leq
	C\|\rho\dot u\|_{L^p}.
\end{align}
Unlike $\omega^{1,3}$ and $\omega^{2,3}$, the component $\omega^{1,2}$ does not satisfy a homogeneous Dirichlet boundary condition on $\partial\Omega$. We therefore use an elliptic estimate only in the tangential variables $(x_1,x_2)$. From \eqref{demo}
with $j=1$, $2$, we have
\begin{align*}
\rho\dot u^1=F_{x_1}+\mu\left(\omega^{1,2}_{x_2}+\omega^{1,3}_{x_3}\right),
\ \
\rho\dot u^2=F_{x_2}+\mu\left(-\omega^{1,2}_{x_1}+\omega^{2,3}_{x_3}\right).
\end{align*}
	Differentiating the first equation with respect to $x_2$, differentiating
	the second one with respect to $x_1$, and subtracting, one gets that
	\begin{align*}
	\mu\Delta_{x_1,x_2}\omega^{1,2}
	=
	\partial_{x_2}\left(\rho\dot u^1-\mu\omega^{1,3}_{x_3}\right)
	-
	\partial_{x_1}\left(\rho\dot u^2-\mu\omega^{2,3}_{x_3}\right),
\end{align*}
	where
	\begin{align*}
	\Delta_{x_1,x_2}\triangleq \partial_{x_1}^{2}+\partial_{x_2}^{2}.
\end{align*}
	For each fixed $x_3>0$, the Calder\'on--Zygmund estimate in
	$\mathbb R^2$ implies
	\begin{align*}
	\|\nabla_{x_1,x_2}\omega^{1,2}(\cdot,\cdot,x_3)\|_{L^p(\mathbb R^2)}
	\leq
	C\Big(
	\|\rho\dot u(\cdot,\cdot,x_3)\|_{L^p(\mathbb R^2)}
	+
	\|\partial_{x_3}\omega^{1,3}(\cdot,\cdot,x_3)\|_{L^p(\mathbb R^2)}
	+
	\|\partial_{x_3}\omega^{2,3}(\cdot,\cdot,x_3)\|_{L^p(\mathbb R^2)}
	\Big).
\end{align*}
	Integrating with respect to $x_3$ and using the estimates for
	$\omega^{1,3}$ and $\omega^{2,3}$, we deduce that
\begin{align*}
	\|\nabla_{x_1,x_2}\omega^{1,2}\|_{L^p}
	\leq
	C\|\rho\dot u\|_{L^p}.
\end{align*}
Since the normal derivative satisfies the identity
$\omega^{1,2}_{x_3}=\omega^{1,3}_{x_2}-\omega^{2,3}_{x_1}$,
one infers from \eqref{j1} and \eqref{j2} that
\begin{align*}
\|\omega^{1,2}_{x_3}\|_{L^p}\leq\|\nabla\omega^{1,3}\|_{L^p}
+\|\nabla\omega^{2,3}\|_{L^p}\leq C\|\rho\dot u\|_{L^p}.
\end{align*}
Combining the tangential and normal derivative estimates, we get
\begin{align*}
\|\nabla\omega^{1,2}\|_{L^p}\leq C\|\rho\dot u\|_{L^p},
\end{align*}
which together with \eqref{j1} and \eqref{j2} indicates that
\begin{align}\label{curl}
\|\nabla\operatorname{curl}u\|_{L^p}\leq C\|\rho\dot u\|_{L^p}.
\end{align}
The estimate for $\nabla F$ then follows from the decomposition
\eqref{demo}. Hence, \eqref{41} is proved.
	
Finally, by	Lemmas \ref{og}, \ref{ho}, and \eqref{41}, we obtain \eqref{40} and
	\begin{align*}
		\|\operatorname{curl}u\|_{L^{p}}+\|F\|_{L^{p}}
		&\leq
		C\|\operatorname{curl}u\|_{L^{2}}^{\frac{6-p}{2p}}
		\|\nabla\operatorname{curl}u\|_{L^{2}}^{\frac{3p-6}{2p}}
		+
		C\|F\|_{L^{2}}^{\frac{6-p}{2p}}
		\|\nabla F\|_{L^{2}}^{\frac{3p-6}{2p}}\\
		&\leq
		C\left(
		\|\nabla u\|_{L^{2}}
		+\|P \|_{L^{2}}
		\right)^{\frac{6-p}{2p}}
		\|\rho\dot u\|_{L^{2}}^{\frac{3p-6}{2p}},
	\end{align*}
	which gives \eqref{42}. It follows from Lemma \ref{ho} and \eqref{42} that
	\begin{align*}
		\|\nabla u\|_{L^{p}}
		&\leq
		C\left(
		\|\operatorname{div}u\|_{L^{p}}
		+\|\operatorname{curl}u\|_{L^{p}}
		\right)\\
		&\leq
		C\left(
		\|F\|_{L^{p}}
		+\|P \|_{L^{p}}
		+\|\operatorname{curl}u\|_{L^{p}}
		\right)\\
		&\leq
		C\|\rho\dot u\|_{L^{2}}^{\frac{3p-6}{2p}}
		\|\nabla u\|_{L^{2}}^{\frac{6-p}{2p}}
		+
		C\|\rho\dot u\|_{L^{2}}^{\frac{3p-6}{2p}}
		\|P \|_{L^{2}}^{\frac{6-p}{2p}}
		+
		C\|P \|_{L^{p}},
	\end{align*}
proving \eqref{43}.
\end{proof}

Next, we give an estimate for $\|\nabla u\|_{L^{\infty}}$, which is
crucial for obtaining the estimate of $\nabla\rho$ in $L^{2}\cap L^{q}$.
\begin{lemma}\label{cinf}
	Let $3<q<6$, and let $(\rho,u)$ be a smooth solution to
	\eqref{1}--\eqref{4}. Then there exists a positive constant $C$ depending only $q$ such that
	\begin{align}\label{g4}
		\|\nabla u\|_{L^{\infty}}
		\leq
		C\big[
		1+
		\ln\left(e+\|\nabla P\|_{L^{q}}\right)
		\big(
		\|P \|_{L^{\infty}}
		+\|P \|_{L^{2}}
		\big)
		\big]+
		C\left(
		\|\rho\dot u\|_{L^{2}}
		+\|\rho\dot u\|_{L^{q}}
		\right).
	\end{align}
\end{lemma}

\begin{proof}
Let $u=v+w$, where $v$ and $w$ satisfy
\begin{align*}
\begin{cases}
\mu\Delta v+(\mu+\lambda)\nabla\operatorname{div}v
=\nabla P,& \text{in } \Omega,\\
v^{3}=\partial_{x_3}v^{1}=\partial_{x_3}v^{2}=0,&
\text{on } \partial\Omega,
\end{cases}
\ \ \begin{cases}
\mu\Delta w+(\mu+\lambda)\nabla\operatorname{div}w
		=\rho\dot u,
		& \text{in } \Omega,\\
w^{3}=\partial_{x_3}w^{1}=\partial_{x_3}w^{2}=0,& \text{on } \partial\Omega.
	\end{cases}
\end{align*}
By Lemma \ref{ellp} and Sobolev's inequality, we obtain that
	\begin{gather}\label{g3}
		\|\nabla v\|_{L^{q}}
		\leq
		C\|P \|_{L^{q}},
		\quad
		\|\nabla^{2}v\|_{L^{q}}
		\leq
		C\|\nabla P\|_{L^{q}},\\
 \label{g5}
		\|\nabla^{2}w\|_{L^{q}}
		\leq
		C\|\rho\dot u\|_{L^{q}},
		\quad
		\|\nabla w\|_{L^{\infty}}
		\leq
		C\left(
		\|\rho\dot u\|_{L^{2}}
		+\|\rho\dot u\|_{L^{q}}
		\right).
	\end{gather}
	
	It remains to estimate $\|\nabla v\|_{L^{\infty}}$. Denote by
	\begin{align*}
	\Omega_{r}(x)=\Omega\cap B_{r}(x),
\end{align*}
	where $B_{r}(x)$ is the ball centered at $x$ with radius $r$.
According to \cite[Lemma 4.9]{Duan2012}, we have
	\begin{align}\label{g0}
		\|\nabla v\|_{L^{\infty}}
		\leq
		C\big[
		1+
		\ln\left(e+\|\nabla^{2}v\|_{L^{q}}\right)
		\|\nabla v\|_{\mathrm{BMO}}
		\big].
	\end{align}
	Here
\begin{align*}
	\|\nabla v\|_{\mathrm{BMO}}
	=
	\|\nabla v\|_{L^{2}}+[\nabla v]_{\mathrm{BMO}},
\end{align*}
	where
\begin{gather*}
	[\nabla v]_{\mathrm{BMO}}
	=
	\sup_{r>0,\ x\in\Omega}
	\frac{1}{|\Omega_{r}(x)|}
	\int_{\Omega_{r}(x)}
	\left|
	\nabla v(y)-(\nabla v)_{\Omega_r(x)}
	\right|\,dy,\ \
	(\nabla v)_{\Omega_r(x)}
	=
	\frac{1}{|\Omega_{r}(x)|}
	\int_{\Omega_{r}(x)}\nabla v(y)\,dy.
\end{gather*}
	By \eqref{g3} and \cite[Proposition 2.2]{SWZ2011}, one finds that
	\begin{align*}
		\|\nabla v\|_{\mathrm{BMO}}
		\leq
		C\left(
		\|P \|_{L^{\infty}}
		+\|P \|_{L^{2}}
		\right),
	\end{align*}
which combined with \eqref{g3} and \eqref{g0} implies that
\begin{align*}
\|\nabla v\|_{L^{\infty}}
\leq C\big[1+\ln\left(e+\|\nabla P\|_{L^{q}}\right)
\left(\|P \|_{L^{\infty}}+\|P \|_{L^{2}}\right)\big].
\end{align*}
This along with \eqref{g5} yields the desired \eqref{g4}.
\end{proof}

Finally, the following Zlotnik's inequality (see \cite[Lemma 1.3]{Z2000}) will be used to derive the time-uniform upper bound of the density.
\begin{lemma}\label{zo}
	Let the function  $y$  satisfy
	\begin{align*}
		y^{\prime}(t)=g(y)+b^{\prime}(t) \text { on }[0, T],\quad  y(0)=y^{0},
	\end{align*}
	with  $g \in C(R)$  and  $y$, $b \in W^{1,1}(0, T)$. If  $g(\infty)=-\infty$  and
	\begin{align*}
		b\left(t_{2}\right)-b\left(t_{1}\right) \leq N_{0}+N_{1}\left(t_{2}-t_{1}\right),
	\end{align*}
	for all  $0 \leq t_{1}<t_{2} \leq T$  with some  $N_{0} \geq 0$  and  $N_{1} \geq 0$, then
	\begin{align*}
		y(t) \leq \max \left\{y^{0}, \bar{\zeta}\right\}+N_{0}<\infty\quad \text { on }\quad [0, T],
	\end{align*}
	where  $\bar{\zeta}$  is a constant such that
	\begin{align*}
		g(\zeta) \leq-N_{1} \quad\text { for }\quad \zeta \geq \bar{\zeta}.
	\end{align*}
\end{lemma}

\section{A \textit{ priori}  estimates}\label{sec3}

In this section we establish global {\it a priori} estimates of the local strong solution $(\rho,u)$ to \eqref{1}--\eqref{4} obtained in Proposition \ref{0pro}. Recalled that
\begin{align*}
	D(t)=\|\nabla u(t)\|_{L^{2}}^{2} +\|P(t)\|_{L^{2}}^{2},\quad
	E(t)= \frac12\|\sqrt{\rho} u(t)\|_{L^{2}}^{2}+\frac{1}{\gamma-1}\|P(\rho(t))\|_{L^{1}}.
\end{align*}

\subsection{A \textit{ priori}  estimates I: scaling-invariant estimates}\label{sub1}

In this subsection, we derive uniform-in-time lower-order estimates by a bootstrap argument. Throughout the rest of this subsection, $C$ denotes a generic positive constant which is independent of the initial data and $T$, and may change from line to line.

We first derive the basic energy estimate.
\begin{lemma}\label{lem1}
It holds that
	\begin{align}\label{x0}
	 \sup _{0 \leq t \leq T}E(t)
		+\int_{0}^{T}  \|\nabla u \|_{L^{2}}^{2}  \, d t\leq  E(0).
		\end{align}
\end{lemma}

\begin{proof}
	Noting that $\nu=-e_3$, hence
	\begin{align}\label{bian3}
		\frac{\partial u}{\partial \nu}=-\partial_{x_3}u.
	\end{align}
	Moreover, the slip boundary condition \eqref{4}
	is equivalent to
	\begin{align*}
		u^3=\partial_{x_3}u^1=\partial_{x_3}u^2=0
		\quad\text{on }\partial\Omega.
	\end{align*}
	Therefore, one has that
	\begin{align}\label{bian1}
		\frac{\partial u}{\partial \nu} \cdot u= \partial_{x_3}u\cdot u
		=
		\partial_{x_3}u^1u^1
		+\partial_{x_3}u^2u^2
		+\partial_{x_3}u^3u^3
		=0
		\quad\text{on }\partial\Omega.
	\end{align}
	
Multiplying \eqref{1}$_{2}$ by $u$, integrating the resultant over $\Omega$, and using the boundary condition \eqref{4} and \eqref{bian1}, we obtain that
\begin{align}\label{kinetic-energy}
	 \frac12\frac{d}{dt}\|\sqrt{\rho} u \|_{L^{2}}^{2}
	+\mu\|\nabla u \|_{L^{2}}^{2}+(\mu+\lambda)\|\operatorname{div}u \|_{L^{2}}^{2}
	=\int_\Omega P\operatorname{div}u\,dx.
\end{align}
Since $P= \rho^\gamma$ with $\gamma>1$, it follows from \eqref{1}$_{1}$ that
\begin{align}\label{PPP}
	\left(\frac{P}{\gamma-1}\right)_t
	+\operatorname{div}\left(\frac{P}{\gamma-1}u\right)
	+P\operatorname{div}u=0.
\end{align}
Integrating \eqref{PPP} over $\Omega$ and using
$(u\cdot\nu)\big|_{\partial\Omega}=0$, we get
\begin{equation*}
\frac{1}{\gamma-1}\frac{d}{dt}\|P \|_{L^{1}}
+\int_\Omega P\operatorname{div}u\,dx=0,
\end{equation*}
which combined with \eqref{kinetic-energy} leads to
\begin{align*}
	\frac{d}{dt}
	\bigg(\frac12 \|\sqrt{\rho} u \|_{L^{2}}^{2} +\frac{1}{\gamma-1} \|P \|_{L^{1}}\bigg)
	+\mu\|\nabla u \|_{L^{2}}^{2}+(\mu+\lambda)\|\operatorname{div}u \|_{L^{2}}^{2}
	=0.
\end{align*}
 	Integrating this equality over $(0,T)$ gives \eqref{x0}.
\end{proof}

\begin{proposition}\label{0pro}
	Assume that the hypotheses of Theorem \ref{th1} hold.
	Then there exist positive constants $\varepsilon_{1}$ and $c_{0}$
	depending only on the physical parameters $\mu$ and $\lambda$ such that if, for all $(x,t)\in \Omega\times[0,T]$,
	\begin{align}\label{cp1}
		\rho(x,t)\leq 2\bar\rho,
		\qquad
			\left( \bar\rho^3E(0)+\bar\rho^{\frac{3-\gamma}{2}}\right) D(t) \leq 2\varepsilon,
	\end{align}
	for some positive constant  $\varepsilon$  satisfying
	\begin{align*}
c_{0}\varepsilon_{0}\leq \varepsilon\leq \varepsilon_{1},
\end{align*}
	then, for all $(x,t)\in\Omega\times[0,T]$,
	\begin{align*}
		\rho(x,t)\leq \frac32\bar\rho,
		\qquad
		\left( \bar\rho^3E(0)+\bar\rho^{\frac{3-\gamma}{2}}\right) D(t)\leq \frac32\varepsilon.
	\end{align*}
		Here $\varepsilon_{1}$ and $c_{0}$  are chosen according to \eqref{sca3}, \eqref{eta-small-new}, \eqref{xsca1}, and \eqref{2.43}.
\end{proposition}

\begin{lemma}\label{lem2}
Under the  assumption \eqref{cp1}, it holds that
	\begin{align}\label{2.15}
		\sup _{0 \leq t \leq T}D(t) +\int_{0}^{T}\left( \|\sqrt{\rho} \dot{u}\|_{L^{2}}^{2}+ \|P\|_{L^{3}}^{3}\right) \, d t \leq CD(0).
	\end{align}
\end{lemma}
\begin{proof}
	Differentiating $u^3=0$ with respect to $t$ gives
$u_t^3=0$ on $\partial\Omega$,
which combined with \eqref{bian3} yields that
	\begin{align}\label{bian2}
	\frac{\partial u}{\partial\nu}\cdot u_t
	=
	-\partial_{x_3}u\cdot u_t
	=
	-\partial_{x_3}u^1u_t^1
	-\partial_{x_3}u^2u_t^2
	-\partial_{x_3}u^3u_t^3
	=0
	\quad\text{on }\partial\Omega.
\end{align}
Multiplying \eqref{1}$_{2}$ by $u_t$, integrating the resultant over $\Omega$, and using $(u_t\cdot\nu)\big|_{\partial\Omega}=0$ and \eqref{bian2}, one has that
 \begin{align}
 	&\frac{1}{2}\frac{d}{dt}
 	\big(
 	\mu\|\nabla u \|_{L^{2}}^{2}+(\mu+\lambda)\|\operatorname{div}u \|_{L^{2}}^{2}
 	\big)
 	+\|\sqrt{\rho} u_{t}\|_{L^{2}}^{2} \notag\\
 	&  	=
 	-\int_\Omega \rho u\cdot\nabla u\cdot u_t\,dx+\frac{d}{d t} \int_\Omega P \operatorname{div} u \, d x
 	-\int_\Omega P_t\operatorname{div}u\,dx.
 	\label{ut-energy-identity-1}
 \end{align}
Owing to
\begin{align*}
P_{t}+\operatorname{div}(Pu)+(\gamma-1)P\operatorname{div}u=0,\quad \operatorname{div}u=\frac{F+P}{2\mu+\lambda},
\end{align*}
one gets that
\begin{align}\label{diu}
\int_\Omega P_t\operatorname{div}u\,dx&=\frac{1}{ 2 \mu+\lambda }\int_\Omega P_tF\,dx+\frac{1}{2(2 \mu+\lambda)} \frac{d}{d t} \|P \|_{L^{2}}^{2}\notag\\
&= \frac{1}{ 2 \mu+\lambda }\int_\Omega Pu\cdot \nabla F\,dx-\frac{\gamma-1}{ 2 \mu+\lambda }\int_\Omega P\operatorname{div}uF \,dx+\frac{1}{2(2 \mu+\lambda)} \frac{d}{d t}\|P \|_{L^{2}}^{2}.
\end{align}
Thus, it follows from $\dot u=u_t+u\cdot\nabla u$, \eqref{ut-energy-identity-1}, and \eqref{diu} that
 \begin{align*}
 	&  \frac{1}{2} \frac{d}{d t}\big(
 	\mu\|\nabla u \|_{L^{2}}^{2}+(\mu+\lambda)\|\operatorname{div}u \|_{L^{2}}^{2}
 	\big) +\|\sqrt{\rho} \dot{u}\|_{L^{2}}^{2} \\
 	&=   \int_\Omega \rho u \cdot \nabla u \cdot \dot{u}\, d x+\frac{d}{d t} \int_\Omega P \operatorname{div} u \, d x-\int_\Omega P_{t} \operatorname{div} u \, d x \\
 	&=  \int_\Omega \rho u \cdot \nabla u \cdot \dot{u} \, d x+\frac{d}{d t} \int_\Omega P \operatorname{div} u \, d x-\frac{1}{2(2 \mu+\lambda)} \frac{d}{d t}\|P \|_{L^{2}}^{2}-\frac{1}{2 \mu+\lambda} \int_\Omega P u \cdot \nabla F \, d x \\
 	&\quad  +\frac{\gamma-1}{2 \mu+\lambda} \int_\Omega P \operatorname{div} u F \, d x\\
 	& \leq \frac{1}{2}\|\sqrt{\rho} \dot{u}\|_{L^{2}}^{2}+C \int_\Omega \rho|u|^{2}|\nabla u|^{2}\, d x+\frac{d}{d t} \int_\Omega P \operatorname{div} u \, d x-\frac{1}{2(2 \mu+\lambda)} \frac{d}{d t} \|P \|_{L^{2}}^{2}\\
 	&\quad -\frac{1}{2 \mu+\lambda} \int_\Omega P u \cdot \nabla F \, d x  +\frac{\gamma-1}{2 \mu+\lambda} \int_\Omega P \operatorname{div} u F \, d x,
 \end{align*}
 which yields that
 \begin{align}\label{2.5}
  \frac{d}{d t}  A(t)+\|\sqrt{\rho} \dot{u}\|_{L^{2}}^{2}  \leq C \int_\Omega \rho|u|^{2}|\nabla u|^{2} \, d x+C\left| \int_\Omega P u \cdot \nabla F \, d x\right| +C\left| \int_\Omega P \operatorname{div} u F \, d x\right|
 	\triangleq \sum_{i=1}^{3} I_{i},
 \end{align}
where
 \begin{align*}
A(t)\triangleq \mu\|\nabla u\|_{L^{2}}^{2}+(\mu+\lambda)\|\operatorname{div} u\|_{L^{2}}^{2}-2\int_\Omega   P \operatorname{div} u \, dx+\frac{1}{ 2 \mu+\lambda }  \|P \|_{L^{2}}^{2}.
 \end{align*}

We need to estimate the terms $I_i$. By \eqref{41}, \eqref{43}, H\"older's
inequality, and \eqref{cp1}, one sees that
\begin{align*}
 	I_{1} & \leq C\|\rho\|_{L^{\infty}}\|u\|_{L^{6}}^{2}\|\nabla u\|_{L^{3}}^{2} \notag\\[1mm]
 	&\leq C \bar{\rho} \|\nabla u\|_{L^{2}}^{2}\left(\| \rho  \dot{u}\|_{L^{2}} \|\nabla u\|_{L^{2}} +\| \rho  \dot{u}\|_{L^{2}}\|P \|_{L^{2}}+ \|P \|_{L^{3}}^{2}\right) \\[1mm]
&\leq C \bar{\rho}^{\frac{3}{2}}\|\nabla u\|_{L^{2}}^{2}\left(\|\nabla u\|_{L^{2}}+ \|P \|_{L^{2}}\right)\|\sqrt{\rho} \dot{u}\|_{L^{2}}+C \bar{\rho}\|\nabla u\|_{L^{2}}^{2}\left\|P\right\|_{L^{3}}^{2}\notag\\
&\leq \frac{1}{4}\|\sqrt{\rho} \dot{u}\|_{L^{2}}^{2} +C \bar{\rho}^{3}\|\nabla u\|_{L^{2}}^{4}\left(\|\nabla u\|_{L^{2}}^{2}+ \|P \|_{L^{2}}^{2}\right)+C \|P \|_{L^{3}}^{3},
\\
 	I_{2}+I_{3}  &  \leq C\|\nabla u\|_{L^{2}} \|P \|_{L^{3}}\|\nabla F\|_{L^{2}} \\[1mm]
 	& \leq C \bar{\rho}^{\frac{1}{2}}\|\nabla u\|_{L^{2}} \|P \|_{L^{3}}  \|\sqrt{\rho} \dot{u}\|_{L^{2}}\\
 	& \leq \frac{1}{4}\|\sqrt{\rho} \dot{u}\|_{L^{2}}^{2} +C \bar{\rho}^{3}\|\nabla u\|_{L^{2}}^{6}+C \|P \|_{L^{3}}^{3}.
\end{align*}
 Substituting the above estimates on $I_{i}$ into \eqref{2.5} leads to
\begin{align}\label{2.12}
  \frac{d}{d t} A(t)+\|\sqrt{\rho} \dot{u}\|_{L^{2}}^{2}  \leq C \bar{\rho}^{3}\|\nabla u\|_{L^{2}}^{4}\big(\|\nabla u\|_{L^{2}}^{2}+\|P\|_{L^{2}}^{2}\big)+C_{1}\|P\|_{L^{3}}^{3}.
\end{align}

Next, we derive an estimate for the pressure. It follows from \eqref{1}$_{1}$ that
	\begin{align}\label{xx1}
		P_t+u\cdot \nabla P+\gamma P \operatorname{div} u=0.
	\end{align}
For  $p \geq 2$, multiplying \eqref{xx1} by  $p P^{p-1}$  and integrating the resultant over  $\Omega$ yields that
	\begin{align*}
		\frac{d}{d t}\|P\|_{L^{p}}^{p} +\frac{p \gamma-1}{2 \mu+\lambda}\|P\|_{L^{p+1}}^{p+1}=-\frac{p \gamma-1}{2 \mu+\lambda} \int_\Omega P^{p} F \, d x.
	\end{align*}
By H\"older's and Young's inequalities, this implies
	\begin{align}\label{xx2}
		\frac{d}{d t}\|P\|_{L^{p}}^{p}+\frac{p \gamma-1}{2(2 \mu+\lambda)}\|P\|_{L^{p+1}}^{p+1} \leq C \|F\|_{L^{p+1}}^{p+1}.
	\end{align}
	Taking $p=2$ in \eqref{xx2}, we get that
	\begin{align}\label{x21}
		\frac{d}{d t}\|P\|_{L^{2}}^{2}+\frac{2 \gamma-1}{2(2 \mu+\lambda)}\|P\|_{L^{3}}^{3}  \leq C \|F\|_{L^{3}}^{3}.
	\end{align}

Choosing a positive constant $C_2>\frac{2(2 \mu+\lambda)(C_{1}+1)}{2 \gamma-1}>0$ sufficiently large such that
	\begin{align*}
		\|\nabla u\|_{L^{2}}^{2}+\|P\|_{L^{2}}^{2} \leq A(t)+C_{2}\|P\|_{L^{2}}^{2} \leq   C\left( \|\nabla u\|_{L^{2}}^{2}+ \|P\|_{L^{2}}^{2}\right).
	\end{align*}
Then, multiplying \eqref{x21} by $C_2$ and adding the resultant to
	\eqref{2.12},  we arrive at
	\begin{align}\label{zui1}
	  \frac{d}{d t} \big(A(t)+C_{2}\|P(t)\|_{L^{2}}^{2}\big)+\|\sqrt{\rho} \dot{u}\|_{L^{2}}^{2}+ \|P\|_{L^{3}}^{3}  \leq C \bar{\rho}^{3}\|\nabla u\|_{L^{2}}^{4}\big(\|\nabla u\|_{L^{2}}^{2}+\|P\|_{L^{2}}^{2}\big) +C\bar{\rho}^{3}  \|P \|_{L^{2}} ^{6},
	\end{align}
	where we have used  Young's inequality, \eqref{42}, and
	\begin{align*}
		 C \|F\|_{L^{3}}^{3} \leq \frac{1}{4}\|\sqrt{\rho} \dot{u}\|_{L^{2}}^{2}+C\bar{\rho}^{3}\left( \|P \|_{L^{2}}^{6}+\|\nabla u \|_{L^{2}}^{6}\right).
	\end{align*}	
	Integrating \eqref{zui1}  over $(0, T)$ and using \eqref{x0}, we obtain taht
	\begin{align}\label{ok}
		&\sup_{0\leq t\leq T}D(t)
		+
		\int_{0}^{T}
		\left(
		\|\sqrt{\rho}\dot u\|_{L^{2}}^{2}
		+
		\|P\|_{L^{3}}^{3}
		\right)\,dt
		\notag\\
		&\leq
		CD(0)
		+
		C\bar\rho^{3}\int_{0}^{T}\|P\|_{L^{2}}^{6}\,dt
		+
		C\bar\rho^{3}
		\int_{0}^{T}D(t)\|\nabla u\|_{L^{2}}^{4}\,dt
		\notag\\
		&\leq
		CD(0)
		+
		C\bar\rho^{3}
		\int_{0}^{T}
		\|P\|_{L^{1}}\|P\|_{L^{2}}^{2}\|P\|_{L^{3}}^{3}\,dt
		+
		C\bar\rho^{3}
		\left( \sup_{0\leq t\leq T}D(t)\right)^{2}
		\int_{0}^{T}\|\nabla u\|_{L^{2}}^{2}\,dt
		\notag\\
		&\leq
		CD(0)
		+
		C\bar\rho^{3}E(0)
		\sup_{0\leq t\leq T}D(t)
		\int_{0}^{T}\|P\|_{L^{3}}^{3}\,dt
		+
		C\bar\rho^{3}E(0)
		\left(
		\sup_{0\leq t\leq T}D(t)
		\right)^2
		\notag\\
		&\leq
		CD(0)
		+
		C\bar\rho^{3}E(0)
		\sup_{0\leq t\leq T}D(t)
		\left(
		\int_{0}^{T}\|P\|_{L^{3}}^{3}\,dt
		+
		\sup_{0\leq t\leq T}D(t)
		\right).
	\end{align}
	Choosing $\varepsilon>0$ in \eqref{cp1} sufficiently small so that
	\begin{align}\label{sca3}
		C\bar\rho^{3}E(0)
		\sup_{0\leq t\leq T}D(t)
		\leq
		\frac{1}{2},
	\end{align}
the last term on the right-hand side of \eqref{ok} can be absorbed into the
	left-hand side. Therefore, \eqref{2.15} follows.
\end{proof}

Now we can prove Proposition \ref{0pro}.
\begin{proof}[Proof of Proposition \ref{0pro}]
	Fix $(x,t)\in \Omega\times[0,T]$, and let $X(s;x,t)$ be the particle
	trajectory determined by
\begin{align*}
\begin{cases}
\frac{d}{ds}X(s;x,t)=u(X(s;x,t),s),\quad 0\leq s<t,\\
X(t;x,t)=x.
\end{cases}
\end{align*}
	For $\delta>0$, define
	\begin{align*}
		\rho^\delta(y,s)
		\triangleq
		\rho(y,s)
		+
		\delta
		\exp\left\{
		-\int_0^s
		\operatorname{div}u(X(\tau;x,t),\tau)\,d\tau
		\right\}.
	\end{align*}
	Then $\rho^\delta>0$, and the continuity equation \eqref{1}$_1$ yields,
	along the trajectory $X(s;x,t)$,
	\begin{align*}
		\frac{d}{ds}\rho^\delta(X(s;x,t),s)
		+
		\rho^\delta(X(s;x,t),s)
		\operatorname{div}u(X(s;x,t),s)
		=0.
	\end{align*}
	
	Set
	\begin{align*}
		Y(s)\triangleq \ln \rho^\delta(X(s;x,t),s).
	\end{align*}
	Since $\operatorname{div}u=\frac{F+P}{2\mu+\lambda}$,
	we obtain that
	\begin{align}\label{Y-eq-new}
		Y'(s)
		=
		-\frac{1}{2\mu+\lambda}
		\left(
		F(X(s;x,t),s)+P(X(s;x,t),s)
		\right)\triangleq =b'(s)+g(Y(s)),
	\end{align}
	where
	\begin{align*}
b(s)\triangleq
		-\frac{1}{2\mu+\lambda}
		\int_0^s F(X(\tau;x,t),\tau)\,d\tau,\ \
g(Y)\triangleq
		-\frac{
			\left(
			e^Y
			-
			\delta
			\exp\left\{
			-\int_0^s
			\operatorname{div}u(X(\tau;x,t),\tau)\,d\tau
			\right\}
			\right)^\gamma
		}{2\mu+\lambda}.
	\end{align*}
	Clearly, $g(Y)\to -\infty$ as $Y\to\infty$.
	
Next, we estimate the oscillation of $b$. The representation formula for
	the \textit{effective viscous flux} gives
	\begin{align*}
		F(X(\tau;x,t),\tau)
		=
		-\frac{d}{d\tau}
		\left[
		(-\Delta)^{-1}\operatorname{div}(\rho u)
		\right](X(\tau;x,t),\tau)
		 +
		[u_i,R_{ij}](\rho u_j)(X(\tau;x,t),\tau),
	\end{align*}
	where
	\begin{align*}
		[u_i,R_{ij}](\rho u_j)
		\triangleq
		u_iR_{ij}(\rho u_j)-R_{ij}(u_i\rho u_j),
		\quad
		R_{ij}\triangleq \partial_i(-\Delta)^{-1}\partial_j.
	\end{align*}
	Therefore, for any $0\leq t_1<t_2\leq T$,
	\begin{align}\label{b-split-new}
		b(t_2)-b(t_1)
	 \leq
		C\sup_{0\leq t\leq T}
		\left\|
		(-\Delta)^{-1}\operatorname{div}(\rho u)(t)
		\right\|_{L^\infty}
		+
		C\int_{t_1}^{t_2}
		\left\|
		[u_i,R_{ij}](\rho u_j)
		\right\|_{L^\infty}\,d\tau
	\triangleq
		\mathrm{III}_1+\mathrm{III}_2.
	\end{align}
By Gagliardo--Nirenberg inequality, Sobolev's inequality, \eqref{cp1}, \eqref{x0}, and \eqref{2.15}, we have
	\begin{align}\label{III1-new}
		\mathrm{III}_1
		&\leq
		C\sup_{0\leq t\leq T}
		\left\|
		(-\Delta)^{-1}\operatorname{div}(\rho u)(t)
		\right\|_{L^6}^{\frac12}
		\left\|
		\nabla(-\Delta)^{-1}\operatorname{div}(\rho u)(t)
		\right\|_{L^6}^{\frac12}
		\notag\\
		&\leq
		C\sup_{0\leq t\leq T}
		\left(
		\|\rho u(t)\|_{L^2}^{\frac12}
		\|\rho u(t)\|_{L^6}^{\frac12}
		\right)
		\notag\\
		&\leq
		C
		\left(
		\bar\rho^3 E(0)
		\sup_{0\leq t\leq T}D(t)
		\right)^{\frac14}
		\notag\\[2mm]
		&\leq
		C(2\varepsilon)^{\frac14}.
	\end{align}
According to the commutator estimate and \eqref{43}, one finds that
	\begin{align}\label{kun}
		\mathrm{III}_2
		&\leq
		C\int_{t_1}^{t_2}
		\|[u_i,R_{ij}](\rho u_j)\|_{L^3}^{\frac15}
		\|\nabla [u_i,R_{ij}](\rho u_j)\|_{L^4}^{\frac45}\,d\tau
		\notag\\
		&\leq
		C\int_{t_1}^{t_2}
		\bar\rho
		\|\nabla u\|_{L^2}
		\|\nabla u\|_{L^6}\,d\tau
		\notag\\
		&\leq
		C\bar\rho^{\frac32}
		\int_{t_1}^{t_2}
		\|\nabla u\|_{L^2}
		\|\sqrt{\rho}\dot u\|_{L^2}\,d\tau
		+
		C\bar\rho
		\int_{t_1}^{t_2}
		\|\nabla u\|_{L^2}
		\|P\|_{L^6}\,d\tau.
	\end{align}
	For the first term on the right-hand side of \eqref{kun}, it follows from H\"older's inequality, \eqref{x0}, and \eqref{2.15} that
	\begin{align}\label{III2-first-new}
	 C\bar\rho^{\frac32}
		\int_{t_1}^{t_2}
		\|\nabla u\|_{L^2}
		\|\sqrt{\rho}\dot u\|_{L^2}\,d\tau
		& \leq
		C\bar\rho^{\frac32}
		\left(
		\int_0^T \|\nabla u\|_{L^2}^2\,d\tau
		\right)^{\frac12}
		\left(
		\int_0^T \|\sqrt{\rho}\dot u\|_{L^2}^2\,d\tau
		\right)^{\frac12}
		\notag\\
		& \leq
		C\bar\rho^{\frac32}
		E(0)^{\frac12}
		\left(
		\sup_{0\leq t\leq T}D(t)
		\right)^{\frac12}
		\notag\\
		& \leq
		C
		\left(
		\bar\rho^3E(0)
		\sup_{0\leq t\leq T}D(t)
		\right)^{\frac12}
		\notag\\[2mm]
		& \leq
		C(2\varepsilon)^{\frac12},
	\end{align}
	while one infers from H{\"o}lder's inequality and
	\eqref{cp1} that
	\begin{align}\label{III2-pressure-new}
	 C\bar\rho\int_{t_1}^{t_2}\|\nabla u\|_{L^2}\|P\|_{L^6}\,d\tau
&\leq C\bar\rho\int_{t_1}^{t_2}\|\nabla u\|_{L^2}\|P\|_{L^\infty}^{\frac23}
		\|P\|_{L^2}^{\frac13}\,d\tau\notag\\
		&\leq
		C\bar\rho^{1+\frac{2\gamma}{3}}
		\int_{t_1}^{t_2}
		D(\tau)^{\frac23}\,d\tau
		\notag\\
		&\leq
		C\bar\rho^{1+\frac{2\gamma}{3}}
		\left(
		\sup_{0\leq t\leq T}D(t)
		\right)^{\frac23}
		(t_2-t_1)
		\notag\\
		&=
		C\bar\rho^\gamma
		\left(
		\bar\rho^{\frac{3-\gamma}{2}}
		\sup_{0\leq t\leq T}D(t)
		\right)^{\frac23}
		(t_2-t_1)
		\notag\\[2mm]
		&\leq
		C\bar\rho^\gamma
		(2\varepsilon)^{\frac23}
		(t_2-t_1).
	\end{align}
	Combining \eqref{III1-new}--\eqref{III2-pressure-new}, we obtain that
	\begin{align*}
		b(t_2)-b(t_1)
		\leq
		C(2\varepsilon)^{\frac14}
		+
		C(2\varepsilon)^{\frac12}+C\bar\rho^\gamma(2\varepsilon)^{\frac23}(t_2-t_1).
	\end{align*}
	
	We choose $\varepsilon>0$ in \eqref{cp1} sufficiently small so that
	\begin{align}\label{eta-small-new}
		C(2\varepsilon)^{\frac23}
		\leq
		\frac{1}{2\mu+\lambda}.
	\end{align}
	Define
	\begin{align*}
		\bar Y_\delta
		\triangleq
		\ln\left\{
		\bar\rho
		+
		\delta
		\exp\left(
		\int_0^T
		\|\operatorname{div}u(\cdot,\tau)\|_{L^\infty}\,d\tau
		\right)
		\right\}.
	\end{align*}
	If $Y\geq \bar Y_\delta$, then
	\begin{align*}
		g(Y)
	\leq
		-\frac{\bar\rho^\gamma}{2\mu+\lambda}
		\leq
		-N_1,
	\end{align*}
	where the last inequality follows from \eqref{eta-small-new}. Thus, the
	hypotheses of Zlotnik's inequality are satisfied for \eqref{Y-eq-new}.
	Therefore,
	\begin{align*}
		\rho^\delta(x,t)
		\leq
		\max\left\{
		\bar\rho+\delta,\,
		\exp\{\bar Y_\delta\}
		\right\}
		\exp\{N_0\}.
	\end{align*}
	Letting $\delta\to0$, we find
	\begin{align*}
		\rho(x,t)
		\leq
		\bar\rho \exp\{N_0\}.
	\end{align*}
	Choosing $\varepsilon>0$ smaller if necessary such that
	\begin{align}\label{xsca1}
		C(2\varepsilon)^{\frac14}
		+
		C(2\varepsilon)^{\frac12}
		\leq
		\ln\frac32,
	\end{align}
	we conclude that
	\begin{align*}
		\rho(x,t)
		\leq
		\bar\rho e^{N_0}
		\leq
		\frac32\bar\rho.
	\end{align*}

Finally, it remains to improve the bound for $D(t)$. By \eqref{2.15}, one has that
	\begin{align*}
		\left(
		\bar\rho^3E(0)+\bar\rho^{\frac{3-\gamma}{2}}
		\right)D(t)
	 \leq
		C\left(
		\bar\rho^3E(0)+\bar\rho^{\frac{3-\gamma}{2}}
		\right)D(0)
	 \leq
		C\varepsilon_0.
	\end{align*}
	Setting
	\begin{align}\label{2.43}
		c_0
		\triangleq
		\max\left\{
		\frac12,\frac{2C}{3}
		\right\},
	\end{align}
then one gets that
	\begin{align*}
	C\varepsilon_0\leq \frac32\varepsilon
\end{align*}
due to $c_0\varepsilon_0\leq \varepsilon$.
Thus, we complete the proof of Proposition \ref{0pro}.
\end{proof}

\subsection{A \textit{ priori}  estimates II: time-weighted estimates}\label{sub2}

Throughout this subsection, we assume that the bootstrap assumption
\eqref{cp1} hold. We use $C(T)$ to emphasize the dependence on $T$. We shall derive several higher-order estimates which are needed to establish the time continuity of the solution. As a direct consequence of the estimates obtained in subsection \ref{sub1}, we
have the following corollary.

\begin{corollary}\label{wen}
	Under the assumption \eqref{cp1},  it holds that
	\begin{gather}
		\sup_{0 \leq t \leq T}\big(\| \sqrt{\rho} u(t)\|_{L^{2}}^{2}+\| \nabla u (t)\|_{L^{2}}^{2}\big)+\int_{0}^{T}\left(\|\nabla u\|_{L^{2}}^{2} +\| \sqrt{\rho}\dot{u}\|_{L^{2}}^{2} \right) \, d t \leq C,\label{wen1}\\
		\sup_{0 \leq t \leq T}\big(\|P(t) \|_{L^{1}} +\|P(t) \|_{L^{2}}^{2}+\|\rho(t) \|_{L^{\infty}}\big)+\int_{0}^{T}\left( \|P \|_{L^{2}}^{2}+ \|P \|_{L^{3}}^{3}\right) \, d t \leq C.\label{wen2}
	\end{gather}
\end{corollary}

\begin{proof}
It suffices to prove the estimate for
$\int_{0}^{T}\|P\|_{L^{2}}^{2}\,dt$ since the remaining estimates follow from subsection \ref{sub1}. Taking the divergence of \eqref{1}$_2$ gives
	\begin{align*}
		P=(-\Delta)^{-1} \operatorname{div}(\rho u)_{t}+(-\Delta)^{-1} \operatorname{div} \operatorname{div}(\rho u \otimes u)+(2 \mu+\lambda) \operatorname{div} u.
	\end{align*}
Multiplying this equality by $P$, integrating the resultant over $\Omega$, and using \eqref{PPP} and \eqref{III1-new}, we obtain that
	\begin{align*}
		\|P\|_{L^{2}}^{2}	\leq & \frac{d}{d t} \int_\Omega(-\Delta)^{-1} \operatorname{div}(\rho u) P \, d x+C\left\|(-\Delta)^{-1} \operatorname{div}(\rho u)\right\|_{L^{\infty}}\|P\|_{L^{2}}\|\nabla u\|_{L^{2}} \notag\\
		& +C\|P\|_{L^{\frac{3}{2}}}\|\nabla u\|_{L^{2}}^{2}+C\|P\|_{L^{2}}\|\nabla u\|_{L^{2}} \notag\\[2mm]
		\leq & \frac{d}{d t} \int_\Omega(-\Delta)^{-1} \operatorname{div}(\rho u) P \, d x+\frac{1}{2}\|P\|_{L^{2}}^{2}+C\|\nabla u\|_{L^{2}}^{2},
	\end{align*}
which implies that
	\begin{align*}
		\int_{0}^{T}\|P\|_{L^{2}}^{2}\,dt
		&\leq
		C\sup_{0\leq t\leq T}
		\left|
		\int_\Omega (-\Delta)^{-1} \operatorname{div}(\rho u)P(t)\,dx
		\right|
		+
		C\int_{0}^{T}\|\nabla u\|_{L^{2}}^{2}\,dt
		\\
		&\leq
		C\sup_{0\leq t\leq T}
		\|(-\Delta)^{-1} \operatorname{div}(\rho u)(t)\|_{L^\infty}
		\sup_{0\leq t\leq T}
		\|P(t)\|_{L^1}
		+
		C\int_{0}^{T}\|\nabla u\|_{L^{2}}^{2}\,dt
		\\
		&\leq C. \tag*{\qedhere}
	\end{align*}
\end{proof}

\begin{lemma}\label{lem52}
Under the assumption \eqref{cp1}, it holds that
\begin{gather}\label{uu0}
\sup _{0 \leq t \leq T}\big( t  \|\nabla u(t)\|_{L^{2}}^{2} +t\|P(t)\|_{L^{2}}^{2}\big)  +\int_{0}^{T} t \left(  \|\sqrt{\rho} \dot{u} \|_{L^{2}}^{2}+\|P\|_{L^{3}}^{3}\right)   \, d t \leq C,\\
\label{ll0}
\sup _{0 \leq t \leq T}\big(  t^{i} \| \sqrt{\rho}\dot{u}(t)\|_{L^{2}}^{2}+t^{i}\| P(t)\|_{L^{3}}^{3}\big) +\int_{0}^{T} t^{i}\left(\|\nabla \dot{u}\|_{L^{2}}^{2}+\|P\|_{L^{4}}^{4}\right) \, d t \leq C, \quad i=1,2.
\end{gather}
\end{lemma}
\begin{proof}
Using \eqref{zui1} and Corollary \ref{wen}, we derive that
	\begin{align}\label{o1}
		\frac{d}{d t}  \big(A(t)+C_{2}\|P\|_{L^{2}}^{2}\big)+\| \sqrt{\rho}\dot{u}\|_{L^{2}}^{2} +\| P\|_{L^{3}}^{3} \leq C  \left(\|\nabla u\|_{L^{2}}^{2}+ \|P \|_{L^{2}}^{2}\right)^{2}.
	\end{align}
	Multiplying \eqref{o1} by $t$ and using Corollary \ref{wen}, \eqref{o1}, and Gronwall's inequality, one gets \eqref{uu0}.
	
	Since $u^3\big|_{\partial\Omega}=0$, its tangential derivatives vanish
	\begin{align*}
		u^3_{x_1}=u^3_{x_2}=0
		\quad \text{on } \partial\Omega.
	\end{align*}
	It follows that
	\begin{align}\label{bian4}
		\dot u^3=
		u_t^3+u\cdot\nabla u^3=0
		\quad \text{on } \partial\Omega.
	\end{align}
	Moreover, differentiating $u^1_{x_3}=u^2_{x_3}=0$ in time and using the  boundary condition \eqref{4}, one also has
	\begin{align}\label{bian5}
		\partial_{x_3}\dot u^1=\partial_{x_3}\dot u^2=0
		\quad \text{on } \partial\Omega.
	\end{align}
	Thus, $\dot u$ satisfies the same homogeneous slip condition on the boundary.
	
Applying the operator $\partial_{t}+\operatorname{div}(u\,\cdot)$ to the $j$-th component of \eqref{1}$_2$ and multiplying the resulting equation by $\dot{u}^{j}$, we obtain  that
	\begin{align}\label{cll5}
		 \frac{1}{2} \frac{d}{d t} \| \sqrt{\rho}\dot{u}\|_{L^{2}}^{2}
		&=\mu \int_\Omega \dot{u}^{j}\big(\partial_{t} \Delta u^{j}+\operatorname{div}\big(u \Delta u^{j}\big)\big) \, d x+(\mu+\lambda) \int_\Omega \dot{u}^{j}\big(\partial_{t} \partial_{j}\operatorname{div} u+\operatorname{div}\big(u \partial_{j}\operatorname{div} u\big)\big) \, d x \notag\\
		&\quad-\int_\Omega \dot{u}^{j}\big(\partial_{j} P_{t}+\operatorname{div} (u \partial_{j}P )\big) \, d x \triangleq \sum_{i=1}^{3} K_{i}.
	\end{align}
We first treat $K_1$. Recall that $\dot u^j=u_t^j+u^k\partial_k u^j$,
then
\begin{align}\label{ex1}
	\Delta \dot u^j
	=
	\Delta\big(u_t^j+u^k\partial_k u^j\big)  &=
	\partial_t\Delta u^j
	+
	\Delta(u^k\partial_k u^j)\notag\\
	&=\partial_t\Delta u^j
	+\partial_i\partial_i(u^k\partial_k u^j) \notag\\
	&=\partial_t\Delta u^j+\partial_i\big(
	\partial_i u^k\partial_k u^j
	+
	u^k\partial_i\partial_k u^j
	\big) \notag\\
	&=\partial_t\Delta u^j+\Delta u^k\partial_k u^j
	+
	2\partial_i u^k\partial_i\partial_k u^j
	+
	u^k\partial_k\Delta u^j.
\end{align}
Moreover, one arrives at
\begin{align*}
	\operatorname{div}(u\Delta u^j)
	=
	\partial_k(u^k\Delta u^j) =
	u^k\partial_k\Delta u^j
	+\operatorname{div}u\Delta u^j,
\end{align*}
which together with \eqref{ex1}
indicates that
\begin{align*}
	\partial_t\Delta u^j
	+
	\operatorname{div}(u\Delta u^j)
	&=\Delta\dot u^j
	-\Delta u^k\partial_k u^j-2\partial_i u^k\partial_i\partial_k u^j+\operatorname{div}u\Delta u^j\notag\\
	&=
	\Delta \dot u^j
	-\partial_i
	\big(
	\partial_i u^k\partial_k u^j
	+\partial_k u^i\partial_k u^j
	-\operatorname{div}u\partial_i u^j
	\big).
\end{align*}
So we have
\begin{align}\label{k1}
	K_1
	=
	\mu\int_\Omega \dot u^j\Delta\dot u^j\,dx
	-\mu\int_\Omega \dot u^j
	\partial_i
	\big(
	\partial_i u^k\partial_k u^j
	+\partial_k u^i\partial_k u^j
	-\operatorname{div}u\partial_i u^j
	\big)\,dx.
\end{align}
By integrating by parts and using boundary condition \eqref{bian4} and \eqref{bian5}, we get that
\begin{align*}
	\mu\int_\Omega \dot u^j\Delta\dot u^j\,dx
	 =
	-\mu\int_\Omega |\nabla\dot u|^2\,dx
	+\mu\int_{\partial\Omega}\dot u^j\frac{\partial\dot u^j}{\partial\nu}\,dS=
	-\mu\int_\Omega |\nabla\dot u|^2\,dx.
\end{align*}
For the remaining terms on the right-hand side of \eqref{k1}, integration by parts yields
\begin{align}\label{kun1}
	&-\mu\int_\Omega \dot u^j
	\partial_i
	\big(
	\partial_i u^k\partial_k u^j
	+\partial_k u^i\partial_k u^j
	-\operatorname{div}u\partial_i u^j
	\big)\,dx\notag\\
	& =
	\mu\int_\Omega \partial_i\dot u^j
	\big(
	\partial_i u^k\partial_k u^j
	+\partial_k u^i\partial_k u^j
	-\operatorname{div}u\partial_i u^j
	\big)\,dx,
\end{align}
where the boundary term vanishes. Indeed, since $\nu=-e_3$, the boundary
term only involves
\begin{align*}
\int_{\partial\Omega}\dot u^j\big( \partial_3 u^k\partial_k u^j
+\partial_k u^3\partial_k u^j
-\operatorname{div}u\partial_3 u^j\big) \,dS.
\end{align*}
For $j=1,2$, this expression is zero on $\partial\Omega$ due to
\begin{align*}
\partial_3u^1=\partial_3u^2=0,
\quad
\partial_1u^3=\partial_2u^3=0
\quad\text{on }\partial\Omega.
\end{align*}
For $j=3$,  $\dot u^3=0$  on $\partial\Omega$.
Consequently, we derive that
\begin{align}
	K_1
	\le
	-\mu\|\nabla \dot{u}\|_{L^{2}}^{2}
	+
	C\int_\Omega |\nabla\dot u|\,|\nabla u|^2\,dx
	\leq -\frac{7\mu}{8}\|\nabla \dot{u}\|_{L^{2}}^{2}
	+C\|\nabla u\|_{L^{4}}^{4}.
	\label{k1-estimate}
\end{align}

We next estimate $K_2$. Owing to
\begin{align*}
	\operatorname{div}\dot u
	=
	\partial_j(u_t^j+u^k\partial_k u^j) =
	\partial_t\operatorname{div}u
	+u\cdot\nabla\operatorname{div}u
	+\partial_j u^k\partial_k u^j,
\end{align*}
we get that
\begin{align}\label{material-div-u}
\partial_t\operatorname{div}u
+u\cdot\nabla\operatorname{div}u
=\operatorname{div}\dot u-\partial_i u^k\partial_k u^i.
\end{align}
Moreover, due to
\begin{align*}
\operatorname{div}\big(u\partial_j\operatorname{div}u\big)
&=u\cdot\nabla\partial_j\operatorname{div}u
+\operatorname{div}u\partial_j\operatorname{div}u,
\end{align*}
one has that
\begin{align}
	\partial_t\partial_j\operatorname{div}u
	+
	\operatorname{div}\big(u\partial_j\operatorname{div}u\big)
	=
	\partial_t\partial_j\operatorname{div}u
	+
	u\cdot\nabla\partial_j\operatorname{div}u
	+
	\operatorname{div}u\partial_j\operatorname{div}u.
	\label{expand-div-term}
\end{align}
We rewrite the first two terms on the right-hand side of
\eqref{expand-div-term}. Since
\begin{align*}
\partial_j\big(u\cdot\nabla\operatorname{div}u\big)
=\partial_j(u^k\partial_k\operatorname{div}u)
=\partial_j u^k\partial_k\operatorname{div}u+u^k\partial_j\partial_k\operatorname{div}u
=\partial_j u^k\partial_k\operatorname{div}u+u\cdot\nabla\partial_j\operatorname{div}u,
\end{align*}
we have
\begin{align*}
u\cdot\nabla\partial_j\operatorname{div}u
=\partial_j\big(u\cdot\nabla\operatorname{div}u\big)
-\partial_j u^k\partial_k\operatorname{div}u.
\end{align*}
Therefore, one sees that
\begin{align}\label{first-two-terms}
\partial_t\partial_j\operatorname{div}u+
u\cdot\nabla\partial_j\operatorname{div}u
=\partial_j\big(\partial_t\operatorname{div}u+u\cdot\nabla\operatorname{div}u\big)
-\partial_j u^k\partial_k\operatorname{div}u.
\end{align}
Substituting \eqref{material-div-u} into \eqref{first-two-terms} leads to
\begin{align*}
	\partial_t\partial_j\operatorname{div}u
	+
	u\cdot\nabla\partial_j\operatorname{div}u
	=
	\partial_j\operatorname{div}\dot u
	-
	\partial_j(\partial_i u^k\partial_k u^i)
	-
	\partial_j u^k\partial_k\operatorname{div}u.
\end{align*}
This along with \eqref{expand-div-term} indicates that
\begin{align*}
 \partial_t\partial_j\operatorname{div}u
	+\operatorname{div}\big(u\partial_j(\operatorname{div}u)\big)=
	\partial_j\operatorname{div}\dot u
	-\partial_j(\partial_i u^k\partial_k u^i)
	+\operatorname{div}u\partial_j\operatorname{div}u
	-\partial_j u^k\partial_k\operatorname{div}u.
\end{align*}
Thus, it follows that
\begin{align}\label{kk2}
	K_2
	&=
	(\mu+\lambda)\int_\Omega \dot u^j\partial_j\operatorname{div}\dot u\,dx
	-(\mu+\lambda)\int_\Omega \dot u^j
	\partial_j(\partial_i u^k\partial_k u^i)\,dx
	\notag\\
	&\quad
	+(\mu+\lambda)\int_\Omega \dot u^j
	\big(
	\operatorname{div}u\partial_j\operatorname{div}u
	-\partial_j u^k\partial_k\operatorname{div}u
	\big)\,dx\notag\\
	&=-(\mu+\lambda)\int_\Omega (\operatorname{div}\dot u)^2\,dx+(\mu+\lambda)\int_\Omega
	\operatorname{div}\dot u\,
	\partial_i u^k\partial_k u^i\,dx\notag\\
	&\quad
	+(\mu+\lambda)\int_\Omega \dot u^j
	\big(
	\operatorname{div}u\partial_j\operatorname{div}u
	-\partial_j u^k\partial_k\operatorname{div}u
	\big)\,dx,
\end{align}
where the boundary term vanishes because $(\dot u\cdot\nu)\big|_{\partial\Omega}=0$.
For the last term on the right-hand side of \eqref{kk2}, we
estimate it directly by splitting it into two parts
\begin{align}
	J
	&\triangleq
	\int_\Omega \dot u^j
	\operatorname{div}u\partial_j\operatorname{div}u\,dx
	-
	\int_\Omega \dot u^j
	\partial_j u^k\partial_k\operatorname{div}u\,dx
	\triangleq J_1+J_2.
\end{align}
For $J_1$, integration by parts together with $(\dot u\cdot\nu)\big|_{\partial\Omega}=0$ gives that
\begin{align}\label{J1-est}
	J_1
	=
	\frac12\int_\Omega \dot u^j
	\partial_j(\operatorname{div}u)^2\,dx
	&=
	-\frac12\int_\Omega
	\operatorname{div}\dot u(\operatorname{div}u)^2\,dx
	+
	\frac12\int_{\partial\Omega}
	\dot u\cdot\nu(\operatorname{div}u)^2\,dS\notag\\
	&=
	-\frac12\int_\Omega
	\operatorname{div}\dot u(\operatorname{div}u)^2\,dx.
\end{align}
 For $J_2$, integrating by parts with respect to $x_k$ yields
\begin{align}\label{J2-est}
	J_2
	&=
	-\int_\Omega \dot u^j
	\partial_j u^k\partial_k\operatorname{div}u\,dx
	\notag\\
	&=
	\int_\Omega
	\partial_k(\dot u^j\partial_j u^k)\operatorname{div}u\,dx
	-
	\int_{\partial\Omega}
	\dot u^j\partial_j u^k\operatorname{div}u\nu_k\,dS
	\notag\\
	&=
	\int_\Omega
	\partial_k\dot u^j\partial_j u^k\operatorname{div}u\,dx
	+
	\int_\Omega
	\dot u^j\partial_j\operatorname{div}u\operatorname{div}u\,dx-
	\int_{\partial\Omega}
	\dot u^j\partial_j u^k\operatorname{div}u\nu_k\,dS\notag\\
	&=
	\int_\Omega
	\partial_k\dot u^j\partial_j u^k\operatorname{div}u\,dx
	+J_{1},
\end{align}
where the boundary term vanishes. Indeed, since $\nu=-e_3$, it only
contains
\begin{align*}
-\dot u^j\partial_j u^3\operatorname{div}u.
\end{align*}
For $j=1,2$, we have $\partial_j u^3=0$ on $\partial\Omega$, while for
$j=3$, we have $\dot u^3=0$ on $\partial\Omega$.
Combining \eqref{kk2}, \eqref{J1-est},  and \eqref{J2-est}, one finds that
\begin{align}
	K_2
	&\le
	-(\mu+\lambda)\int_\Omega (\operatorname{div}\dot u)^2\,dx
	+
	C\int_\Omega |\nabla\dot u|\,|\nabla u|^2\,dx\notag\\
	&\leq
	-\frac{\mu+\lambda}{2}\|\operatorname{div}\dot{u}\|_{L^{2}}^{2}
	+\frac{\mu}{8}\|\nabla\dot{u}\|_{L^{2}}^{2}
	+C\|\nabla u\|_{L^{4}}^{4}.
	\label{I2-estimate}
\end{align}

Now we estimate the last term $K_3$. We obtain from integration by parts that
	\begin{align*}
		K_3
		&=
		-\int_{\Omega}\dot u^j
		\big(
		\partial_jP_t+\operatorname{div}(u\partial_jP)
		\big)\,dx\\
		&=
		-\int_{\Omega}\dot u^j\partial_jP_t\,dx
		-
		\int_{\Omega}\dot u^j\partial_k(u^k\partial_jP)\,dx\\
		&=
		\int_{\Omega}P_t\partial_j\dot u^j\,dx
		+
		\int_{\Omega}u^k\partial_jP\,\partial_k\dot u^j\,dx,
	\end{align*}
	where the boundary terms vanish due to boundary condition  $u\cdot\nu=\dot u\cdot\nu=0$ on $\partial\Omega$. Integrating by parts with respect to $x_j$ leads to
	\begin{align}\label{kk3}
		\int_{\Omega}u^k\partial_jP\,\partial_k\dot u^j\,dx
		 =
		-\int_{\Omega}P\partial_j(u^k\partial_k\dot u^j)\,dx  =
		-\int_{\Omega}P\partial_j u^k\,\partial_k\dot u^j\,dx
		-\int_{\Omega}Pu^k\partial_k\operatorname{div}\dot u\,dx.
	\end{align}
	The corresponding boundary term is
	\begin{align*}
		\int_{\partial\Omega}P u^k\partial_k\dot u^j\nu_j\,dS.
	\end{align*}
	Since $\nu=-e_3$, this term only involves $j=3$. On the flat boundary, $\dot u^3=0$. Hence, its tangential derivatives vanish.
	Together with $u^3=0$ on $\partial\Omega$, we get
	$u^k\partial_k\dot u^3=0$ on $\partial\Omega$. Therefore, the boundary term vanishes. For the last term in \eqref{kk3}, integrating by parts with respect to $x_k$ gives
	\begin{align*}
		-\int_{\Omega}Pu^k\partial_k\operatorname{div}\dot u\,dx
		=
		\int_{\Omega}\partial_k(Pu^k)\operatorname{div}\dot u\,dx=
		\int_{\Omega}
		\left(u\cdot\nabla P+P\operatorname{div}u\right)
		\operatorname{div}\dot u\,dx,
	\end{align*}
	where the boundary term vanishes again because $u\cdot\nu=0$ on $\partial\Omega$.
	Consequently,
	\begin{align}\label{k3}
		K_3
		&=
		\int_{\Omega}
		\left(P_t+u\cdot\nabla P+P\operatorname{div}u\right)
		\operatorname{div}\dot u\,dx
		-\int_{\Omega}P\partial_j u^k\,\partial_k\dot u^j\,dx \notag \\
&=-(\gamma-1)\int_{\Omega}P\operatorname{div}u\,\operatorname{div}\dot u\,dx-\int_{\Omega}P\partial_j u^k\,\partial_k\dot u^j\,dx \notag \\
&\leq
		C\int_{\Omega}P|\nabla u||\nabla\dot u|\,dx\notag\\
		&\leq
		\frac{\mu}{4}\|\nabla\dot u\|_{L^2 }^2
		+
		C\int_{\Omega}P^2|\nabla u|^2\,dx\notag\\
		&\leq
		\frac{\mu}{4}\|\nabla\dot u\|_{L^2 }^2
		+
		C\left( \|P \|_{L^{4}}^{4}+\|\nabla u \|_{L^{4}}^{4}\right)
	\end{align}
due to	\begin{align*}
		P_t+u\cdot\nabla P+P\operatorname{div}u
		=
		-(\gamma-1)P\operatorname{div}u.
	\end{align*}
	
	Substituting the estimates \eqref{k1-estimate}, \eqref{I2-estimate}, and \eqref{k3} into \eqref{cll5} implies
	\begin{align}\label{ll6}
		\frac{d}{d t}\|\sqrt{\rho} \dot{u}\|_{L^{2}}^{2}+ \|\nabla \dot{u}\|_{L^{2}}^{2}  \leq C\left( \|P \|_{L^{4}}^{4}+\|\nabla u \|_{L^{4}}^{4}\right).
	\end{align}	
	Multiplying \eqref{ll6} by $t$ and using \eqref{40}, we get
	\begin{align}\label{cll7}
	 \frac{d}{d t} \left( t\|\sqrt{\rho} \dot{u}\|_{L^{2}}^{2}\right) +  t \|\nabla \dot{u}\|_{L^{2}}^{2}  \leq \|\sqrt{\rho} \dot{u}\|_{L^{2}}^{2}+C t \left( \|F\|_{L^{4}}^{4}+  \|\operatorname{curl} u\|_{L^{4}}^{4}\right) +C_{3} t \|P\|_{L^{4}}^{4}.
	\end{align}
	Choosing a sufficiently large constant $C_{4}>C_{3}+1$.	Taking $p=3$ in \eqref{xx2}, adding $C_{4} t$ times the resulting inequality
	to \eqref{cll7}, and using \eqref{42}, we obtain that
	\begin{align*}
		&\frac{d}{d t}\left(t \|\sqrt{\rho} \dot{u}\|_{L^{2}}^{2}+C_{4} t \|P\|_{L^{3}}^{3}\right) + t\|\nabla \dot{u}\|_{L^{2}}^{2}+t\|P\|_{L^{4}}^{4} \\
		& \leq C \left(\|\sqrt{\rho} \dot{u}\|_{L^{2}}^{2}+\|P\|_{L^{3}}^{3} \right)+C t\left( \|F\|_{L^{4}}^{4}+ \|\operatorname{curl} u\|_{L^{4}}^{4}\right)  \\[1mm]
		& \leq  C \left(\|\sqrt{\rho} \dot{u}\|_{L^{2}}^{2}+\|P\|_{L^{3}}^{3} \right)+C t\|\sqrt{\rho} \dot{u}\|_{L^{2}}^{2}\left(\|\sqrt{\rho} \dot{u}\|_{L^{2}}^{2}+\|\nabla u\|_{L^{2}}^{2}+\|P\|_{L^{2}}^{2}\right),
	\end{align*}
which together with Gronwall's inequality gives \eqref{ll0} with $i=1$.
Similarly, multiplying \eqref{ll6} by $t^{2}$ and adding $C_{4} t^{2}$ times \eqref{xx2} with $p=3$, we can derive \eqref{ll0} with $i=2$.
\end{proof}

\begin{lemma}\label{lem53}
	Under the assumption \eqref{cp1}, there exists a positive
	constant $C(T)$  such that
	\begin{gather*}
		\sup _{0 \leq t \leq T} \left( t \| \nabla^{2} u(t)\|_{L^{2}}^{2}+\|  \rho_{t} (t) \|_{L^{2}}^{2}\right) +\int_{0}^{T} t\| \nabla u_{t}  \|_{L^{2}}^{2}\, d t \leq C(T),\\
		\sup _{0 \leq t \leq T}\left( \|\nabla \rho(t)\|_{L^{2}}+\|\nabla \rho(t)\|_{L^{q}}\right) +\int_{0}^{T}\left( \|\nabla^{2} u \|_{L^{q}}+ \|\nabla^{2} u \|_{L^{2}}+\|\nabla u\|_{L^{\infty}}\right)\, d t \leq C(T).
	\end{gather*}
\end{lemma}

\begin{proof}
	For $p \in[2,6]$, $|\nabla \rho|^{p}$  satisfies
	\begin{align*}
	 \left(|\nabla \rho|^{p}\right)_{t}+\operatorname{div}\left(|\nabla \rho|^{p} u\right)+(p-1)|\nabla \rho|^{p} \operatorname{div} u
	 +p|\nabla \rho|^{p-2}(\nabla \rho)^{t} \nabla u(\nabla \rho)+p \rho|\nabla \rho|^{p-2} \nabla \rho \nabla \operatorname{div} u=0,
	\end{align*}
which combined with Lemma \ref{cinf} yields that
	\begin{align}\label{b3}
		\frac{d}{d t}\|\nabla \rho\|_{L^{q}} &\leq C\|\nabla^{2} u \|_{L^{q}}+C\left( \|\nabla u\|_{L^{\infty}}+1\right) \|\nabla \rho\|_{L^{q}}\notag\\
		&\leq C\|\sqrt{\rho}\dot{u} \|_{L^{q}}+C\left( \|\nabla u\|_{L^{\infty}}+1\right)\|\nabla \rho\|_{L^{q}}\notag\\
		&\leq  C\big(1+\|\sqrt{\rho} \dot{u}\|_{L^{2}}+\|\sqrt{\rho} \dot{u}\|_{L^{q}}+\ln \left(e+\|\nabla \rho\|_{L^{q}}\right)\big)\|\nabla \rho\|_{L^{q}}+C\|\sqrt{\rho} \dot{u}\|_{L^{q}}.
	\end{align}
Here we have used the elliptic estimate for \eqref{1}$_2$:
	\begin{align}\label{ell}
		\|\nabla^{2}u\|_{L^{q}}
		\leq
		C\left(
		\|\sqrt{\rho}\dot{u}\|_{L^{q}}
		+\|\nabla P\|_{L^{q}}
		\right).
	\end{align}

	It follows from \eqref{b3}   that
	\begin{align}\label{lll5}
		\frac{d}{d t} \ln \left( e+\|\nabla \rho\|_{L^{q}}\right)  \leq C \ln \big(e+\|\nabla \rho\|_{L^{q}}\big)\left( 1+\|\sqrt{\rho} \dot{u}\|_{L^{2}}+\| \sqrt{\rho}\dot{u} \|_{L^{q}}\right).
	\end{align}
	Set
	\begin{align*}
	\sigma(T)\triangleq \min\{1,T\}.
\end{align*}
	By Corollary \ref{wen} and Lemma \ref{lem52}, we have
	\begin{align*}
		\int_{0}^{T}\|\sqrt{\rho}  \dot{u} \|_{L^{q}}\, d t
		&\leq C\left(\int_{0}^{\sigma(T)}\| \sqrt{\rho} \dot{u} \|_{L^{2}}^{2}\, d t\right)^{\frac{6-q}{4 q}}\left(\int_{0}^{\sigma(T)} t \| \nabla \dot{u} \|_{L^{2}}^{2} \, d t\right)^{\frac{3 q-6}{4 q}}\left(\int_{0}^{\sigma(T)} t^{-\frac{3 q-6}{2 q}}\, d t\right)^{\frac{1}{2}} \\
		&\quad+C\left(\int_{\sigma(T)}^{T}t \| \sqrt{\rho} \dot{u} \|_{L^{2}}^{2} \, d t\right)^{\frac{6-q}{4 q}}\left(\int_{\sigma(T)}^{T} t^{2} \| \nabla \dot{u}  \|_{L^{2}}^{2} \, d t\right)^{\frac{3 q-6}{4 q}}\left(\int_{\sigma(T)}^{T} t^{-\frac{5q-6}{2q}}\, d t\right)^{\frac{1}{2}}\\[2mm]
		&\leq C.
	\end{align*}
This implies that
	\begin{align}\label{lll7}
		\int_{0}^{T}\left( 1+\|\sqrt{\rho} \dot{u}\|_{L^{2}}+\| \sqrt{\rho}\dot{u} \|_{L^{q}}\right)\, d t \leq C(T).
	\end{align}
Applying Gronwall's inequality to \eqref{lll5} and using \eqref{lll7} indicate that
	\begin{align}\label{v1}
		\sup _{0 \leq t \leq T}\|\nabla \rho\|_{L^{q}} \leq C(T).
	\end{align}
Then Lemma \ref{cinf} gives
	\begin{align}\label{v2}
		\int_{0}^{T}\|\nabla u\|_{L^{\infty}} d t \leq C(T).
	\end{align}	
	By a similar argument, one has that
	\begin{align}\label{v3}
		\sup _{0 \leq t \leq T}\|\nabla \rho\|_{L^{2}} \leq C(T).
	\end{align}
	
It remains to estimate the time derivative of $\rho$ and the higher
spatial derivatives of $u$. From the continuity equation \eqref{1}$_1$,
the elliptic estimate \eqref{ell}, and Gagliardo--Nirenberg
inequality, we obtain that
	\begin{align*}
		\|\rho_{t} \|_{L^{2}}  &\leq C\big( \|\nabla \rho\|_{L^{3}}\|u\|_{L^{6}}+ \|\rho\|_{L^{\infty}}\|\nabla u\|_{L^{2}}\big)  \leq C\|\nabla u\|_{L^{2}}, \\
		\|  \nabla^{2} u \|_{L^{2}} +\|\nabla^{2} u \|_{L^{q}}  &\leq C\left( \| \sqrt{\rho} \dot{u}\|_{L^{2}}+\| \sqrt{\rho} \dot{u} \|_{L^{q}}+\|\nabla \rho\|_{L^{2}}+\|\nabla \rho\|_{L^{q}}\right),\\
\|\nabla u_{t} \|_{L^{2}}  &\leq C\left( \|\nabla \dot{u}\|_{L^{2}}+ \|u\|_{L^{\infty}} \|\nabla^{2} u \|_{L^{2}}+ \|\nabla u\|_{L^{4}}^{2}\right)  \\& \leq C\Big( \|\nabla \dot{u}\|_{L^{2}}+ \|\nabla u\|_{L^{2}}^{\frac{1}{2}} \| \nabla^{2} u \|_{L^{2}}^{\frac{3}{2}}\Big)\\
& \leq C\Big( \|\nabla \dot{u}\|_{L^{2}}+ \|\nabla u\|_{L^{2}}^{\frac{1}{2}} \left( \| \sqrt{\rho} \dot{u}  \|_{L^{2}}+\|\nabla \rho\|_{L^{2}}\right) ^{\frac{3}{2}}\Big).
	\end{align*}
	Combining these estimates with \eqref{v1}--\eqref{v3} and Lemma
	\ref{lem52}, we obtain the desired bounds.
\end{proof}

\section{Proof of Theorem \ref{th1}}\label{sec4}

We are now in a position to prove Theorem \ref{th1}. By the local
well-posedness result in Proposition \ref{local}, there exists a time
$T_{*}>0$ such that the initial-boundary value problem
\eqref{1}--\eqref{4} admits a unique strong solution $(\rho,u)$ on
$\Omega\times(0,T_{*}]$. We first verify that the bootstrap assumptions are valid for a short time.
Indeed, by \eqref{xz}, \eqref{sca}, and the choice of $c_0$ in
\eqref{2.43}, the initial data satisfy \eqref{cp1} at $t=0$.
Therefore, by continuity, there exists $T_1\in(0,T_*]$ such that
\eqref{cp1} holds on $[0,T_1]$.

Define
\begin{align*}
	T^{*}\triangleq \sup \left\{ T>0 \,\middle|\,
	(\rho, u) \text{ is a strong solution of } \eqref{1}-\eqref{4}
	\text{ on } \Omega\times[0,T] \text{ satisfying } \eqref{cp1}
	\right\}.
\end{align*}
Clearly, $T^{*}\geq T_1>0$.   For any $0<\tau<T\leq T^{*}$ with $T<\infty$, Lemma
\ref{lem53} gives
\begin{align}\label{ccon1}
	u\in C\bigl([\tau,T];D^1\bigr).
\end{align}
In addition, the transport equation and the uniform estimates yield
\begin{align}\label{ccon2}
	\rho u\in C\bigl([0,T];L^2\bigr),
	\qquad
	\rho\in C\bigl([0,T];H^1\cap W^{1,q}\bigr).
\end{align}
If $T^{*}<\infty$, the limit
\begin{align*}
	(\rho,u)(T^{*})
	=
	\lim_{t\to T^{*}}(\rho,u)(t)
\end{align*}
holds. The limiting data at time $T^{*}$ also satisfy the  condition
required in \eqref{sou}.
Taking
$(\rho,u)(T^{*})$ as a new initial data, then Proposition \ref{local} yields
a strong solution on $[T^{*},T^{*}+\delta]$ for some $\delta>0$.  This contradicts
the definition of $T^{*}$. Therefore, $T^{*}=\infty$, and the strong solution exists globally in
time. The uniqueness follows from the local uniqueness and the continuation
argument above. This completes the proof of Theorem \ref{th1}. \hfill$\square$

\section*{Conflict of interest}
The authors declare that they have no conflicts to disclose.

\section*{Data availability}
No data was used for the research described in the article.



\end{document}